\documentclass[winfonts,hyperref,a4paper]{article}
\usepackage{mathrsfs}\usepackage{amsmath,amssymb,amsthm,amsfonts,latexsym,graphicx,subfigure}
\usepackage{amscd}\usepackage[numbers,sort&compress]{natbib}
\usepackage{amsmath}\usepackage{bbm}
\usepackage{amssymb}
\usepackage{amsthm}
\usepackage{graphicx}\usepackage{enumerate}
\usepackage{tikz}
\usepackage[left=2.54cm,right=2.54cm,top=3.17cm,bottom=3.17cm]{geometry}
\usepackage{listings}
\theoremstyle{plain}
\newtheorem{defn}{Definition}
\newtheorem{theorem}{Theorem}
\newtheorem{lemma}{Lemma}
\newtheorem{pro}{Proposition}

\newtheorem{cor}{Corollary}
\newtheorem{example}{Example}
\newtheorem{remark}{Remark}

\renewcommand{\vec}[1]{\mbox{\boldmath$#1$}}

\newcommand{\B}{\ensuremath{\mathcal{B}}}

\newcommand{\M}{\ensuremath{\mathcal{M}}}

\newcommand{\vol}{\ensuremath{\mathrm{vol}}}
\newcommand{\diag}{\ensuremath{\mathrm{diag}}}
\newcommand{\card}{\ensuremath{\mathrm{card}}}
\newcommand{\sgn}{\ensuremath{\mathrm{Sgn}}}
\newcommand{\Span}{\ensuremath{\mathrm{span}}}
\newcommand{\tm}{\ensuremath{\mathrm{tm}}}
\newcommand{\am}{\ensuremath{\mathrm{am}}}
\newcommand{\sign}{\ensuremath{\mathrm{sign}}}
\newcommand{\conv}{\ensuremath{\mathrm{conv}}}
\newcommand{\ie}{\ensuremath{\mathrm{i.e~}}}

\newcommand{\argmin}{\ensuremath{\mathrm{argmin}\,}}
\newcommand{\argmax}{\ensuremath{\mathrm{argmax}\,}}
\newcommand{\plug}{\ensuremath{\,{\textcircled{\small{p}}}\,}}
\newcommand{\paste}{\ensuremath{\,\mathop{\sim}\limits^p\,}}
\newcommand{\join}{\ensuremath{\,\mathop{\sim}\limits^j\,}}
\newcommand{\cir}{\ensuremath{{\, \footnotesize \textcircled{}\,}}}
\newcommand{\cirplus}{\ensuremath{{\, \footnotesize\textcircled{\scriptsize{$+$}}\,}}}
\newcommand{\cirminus}{\ensuremath{{\, \footnotesize\textcircled{\scriptsize{$-$}}\,}}}
\title{Nodal Domain and the Multiplicity for Eigenvectors of 1-Laplacian on Graphs}
\begin{document}
\bibliographystyle{unsrt}

\title{Nodal Domains of Eigenvectors for $1$-Laplacian on Graphs}
\author{K.C. Chang\footnotemark[2],
\and Sihong Shao\footnotemark[2],
\and Dong Zhang\footnotemark[2]}
\renewcommand{\thefootnote}{\fnsymbol{footnote}}
\footnotetext[2]{LMAM and School of Mathematical Sciences, Peking University, Beijing 100871, China.}
\date{}
\maketitle
\begin{abstract}
The eigenvectors for graph $1$-Laplacian possess some sort of localization property: On one hand, any nodal domain of an eigenvector is again an eigenvector with the same eigenvalue; on the other hand, one can pack up an eigenvector for a new graph by several fundamental eigencomponents and modules with the same eigenvalue via few special techniques. The Courant nodal domain theorem for graphs is extended to graph $1$-Laplacian for strong nodal domains, but for weak nodal domains it is false. The notion of algebraic multiplicity is introduced in order to provide a more precise estimate of the number of independent
eigenvectors. A positive answer is given to a question raised in [{\sl K.~C. Chang, Spectrum of the $1$-Laplacian and Cheeger constant on graphs, J. Graph Theor., DOI: 10.1002/jgt.21871}], to confirm that the critical values obtained by the minimax principle may not cover all eigenvalues of graph $1$-Laplacian.
\end{abstract}

\section{Introduction}

The spectral theory for $1$-Laplacian on graphs is an interesting object. It has been studied recently in \cite{HeinBuhler2010, Chang2015, ChangShaoZhang2015} etc. Hein and B\"{u}hler \cite{HeinBuhler2010} proposed a straightforward nonlinear generalization of the linear eigenvalue problem for the standard graph Laplacian by defining the graph $1$-Laplacian as
\begin{equation}\label{eq:1-Lap-i}
\textstyle
(\Delta_1 \vec x)_i := \left\{\left.\sum_{j\sim i} z_{ij}(\vec x)\right|z_{ij}(\vec x)\in \sgn(x_i-x_j),\; z_{ji}(\vec x)=-z_{ij}(\vec x),\; \forall j\sim i\right\}\,\,\,\,\,\forall\,i=1,2,\cdots,n,
\end{equation}
where $\vec x=(x_1, \cdots, x_n)\in\mathbb{R}^n$, $j\sim i$ denotes vertex $j$ being adjacent to vertex $i$, $\sum_{j\sim i}$ means the summation over all vertices adjacent to vertex $i$, and
$$
\sgn (t)=\begin{cases}
1, &t>0,\\
-1,&t<0,\\
[-1,1],&t=0.
\end{cases}
$$
In consequence, the corresponding $1$-Laplacian eigenvalue problem is to solve
a pair $(\mu, \vec x)\in \mathbb{R}^1\times (\mathbb{R}^n\backslash\{0\})$  satisfying
\begin{equation}\label{eq:pair}
\vec 0\in \Delta_1 \vec x - \mu D \sgn(\vec x),
\end{equation}
where the diagonal matrix $D=\diag(d_1, \cdots, d_n)$ with $d_i$ being the degree of the vertex $i$,
and $\sgn(\vec x)=(\sgn (x_1), \sgn (x_2),\cdots,\sgn (x_n))^T$.

Very recently, Chang \cite{Chang2015} studied the same problem from a variational point of view and  developed a critical point theory for the function
\begin{equation}\label{eq:I}
I(\vec x)=\sum_{i\sim j}| x_i-x_j|,
\end{equation}
on a piecewise linear manifold
\begin{equation}\label{eq:X}
\textstyle
X=\left\{\vec x\in \mathbb{R}^n: \sum_{i=1}^nd_i|x_i|=1\right\}.
\end{equation}
He proved that the Euler equation of $I(\vec x)$ on $X$ is nothing but the system \eqref{eq:pair}, which can be rewritten in the coordinate form as
\begin{equation}\label{eq:eigen-system}
\left\{\begin{array} {l}
\sum_{j\sim i}z_{ij}\in \mu d_i \sgn(x_i),\,\,\,\,i=1,2,\cdots, n,\\
z_{ij}\in \sgn(x_i-x_j),\\
z_{ji}=-z_{ij}.
\end{array}\right.
\end{equation}
Based on such solid foundation, the Liusternik-Schnirelmann theory was extended, provided both the function $I(\vec x)$ and the constraint $X$ being invariant under the transformation
$\vec x\mapsto (-\vec x)$. By introducing the Krasnoselski genus $\gamma(A)$ of a closed symmetric subset $A\subset \mathbb{R}^n\backslash \{0\}$, a series of critical values are obtained \cite{Chang2015}:
\begin{equation}\label{eq:def-ck}
c_k=\inf_{\gamma(A)\ge k}\max_{\vec x\in A}I(\vec x),\,\,k=1,2,\cdots, n.
\end{equation}
Meanwhile, the Liusternik-Schnirelmann multiplicity theorem was also extended to $I|_X$, in which the topological multiplicity of an eigenvalue was defined via the genus.

More interestingly, the relationship between eigenvalues and the Cheeger constant for graphs was further revealed in \cite{Chang2015}. Let $h(G)$ be the Cheeger constant of the graph $G$, it was proved there that:
$$
c_2=h(G).
$$
And the mountain pass characterization of the Cheeger constant were also obtained.
Actually, the Cheeger constant can be characterized as a minimum of the function $I(\vec x)$ on a feasible subset $\pi$ (see Eq.~\eqref{eq:pi}) of $X$. Precisely, it was shown in \cite{ChangShaoZhang2015}  that the feasible set $\pi$ is nothing but the set of vectors whose median is zero in $X$ (\cite{ChangShaoZhang2015}, Theorem 2.9). Combining with some other ideas like the relaxation, several efficient algorithms were introduced in numerical computations \cite{ChangShaoZhang2015}.

Further more, the nodal domain theorem (\cite{Chang2015}, Theorem 3.6) reveals the structure of eigenvectors for the graph $1$-Laplacian. It implies that in most cases the set of eigenvectors of a given eigenvalue could be very huge, it appears as a union of cells. This huge set of eigenvectors makes trouble in numerical computations.

This paper continues the study. A {\sl new feature} of eigenvectors for the graph $1$-Laplacian is discovered. The eigenvectors possess some sort of localization property: On one hand (decomposition), any nodal domain of an eigenvector is again an eigenvector with the same eigenvalue; on the other hand (package), one can pack up an eigenvector for a new graph by several fundamental eigencomponents and modules with the same eigenvalue via several special techniques. Section \ref{sec:package} is devoted to this study. Theorem \ref{th:equivalent-binary} is on the decomposition, and the rest of this section is on the package of those components. Special techniques, including  extension, joining, pasting and plugging,
are introduced. Examples are provided in illustrating how to apply these techniques.

Based on Theorem \ref{th:equivalent-binary}, for any eigenvalue there must be a corresponding eigenvector with only one nodal domain. In this case, the notion on the largest number $\nu(\mu, G)$ of nodal domains for an eigenvalue $\mu$ on graph $G$ is introduced (see Definition \ref{def:nu}). With the aid of Theorem \ref{th:equivalent-binary} and those special techniques, we are able to calculate  $\nu(\mu, G)$
directly for three special graphs: path graphs $P_n$, cycle graphs $C_n$ and complete graphs $K_n$ in Section \ref{sec:PnCnKn}. It is interesting to note that the results for $P_n$ and $C_n$, stated in Theorems \ref{th:nodal-Pn} and \ref{th:nodal-Cn}, respectively, induce the counterparts of the Sturm-Liouville oscillation theorem in ordinary differential equations (ODE), of the oscillatory eigenfunctions for $1$-Laplacian on intervals and circles \cite{Chang2009}, as well as of the oscillatory eigenfunctions for standard Laplacian on $P_n$ in the linear spectral graph theory \cite{GantmacherKrein2002}.

To an eigenvector $\vec\phi$ on a graph, there are two kinds of nodal domains: the strong nodal domain and the weak nodal domain. They are denoted by $S(\vec\phi)$ and $W(\vec\phi)$, respectively.
The Courant nodal domain theorem has been extended to the standard linear Laplacian on graphs \cite{DaviesGladwellLeydoldStadler2001}. The extended version to the nonlinear Laplacian asserts that $S(\vec\phi_k)\le  k+r-1$, where $\vec\phi_k$ is any eigenvector with eigenvalue $c_k$ and $r$ is the topological multiplicity of $c_k$. This is Theorem \ref{th:strong-nodal-domain}. But for weak nodal domains there is no better estimate, see Example \ref{exam:10G-weak-nodal}. Not like in the linear spectral theory that the eigenvector of the first nonzero eigenvalue must be changing sign, Example \ref{exam:5G-one-nodal} provides a graph with $S(\vec\phi_2)=1$ for the graph $1$-Laplacian. We also give the condition under which $S(\vec\phi_2)\le 2$, see Theorem \ref{th:simple-c2}.

The $k$-way Cheeger constant is usually defined to be
\begin{equation}
\label{eq:hk}
h_k=\min_{S_1, S_2, \cdots, S_k}\max_{1\le i\le k}\frac{|\partial S_i|}{\vol(S_i)},
\end{equation}
where $\{S_1, S_2, \cdots, S_k\}$ is a set of disjoint subsets of $V$. It is motivated by the multi-way spectral partition, see \cite{LeeGharanTrevisan2011} for instance.
Recently, the relationship between the higher-order eigenvalue $\lambda_k$ of the standard graph Laplacian and $h_k$ has been obtained \cite{LeeGharanTrevisan2011}. In Theorem \ref{th:ckhk}, on one hand, we provide an estimate from below: $c_k\le h_k$ for the graph $1$-Laplacian; on the other hand, with the aid of nodal domains, we obtain an estimate from above: $h_m\le c_k$, if $\nu(c_k, G)\ge m$. The details are delineated in Sections \ref{sec:nodal-domain} and \ref{sec:k-way-Cheeger}.

In the rest of this paper, we study  several related topics.

There are two groups of eigenvalues, one by solutions $\{\mu_1, \mu_2, \cdots \}$ of the system \eqref{eq:pair} with natural order, another by $c_1\le c_2\le \cdots \le c_n$ via the minimax principle \eqref{eq:def-ck}. It was then asked \cite{Chang2015}: {\sl Is there any eigenvalue $\mu$, which is not in the sequence: $c_1\le c_2\le \cdots \le c_n$?} Section \ref{sec:example} gives a positive answer.

So far, the multiplicity of an eigenvalue $c$ means the topological multiplicity, i.e., $\gamma(\mathcal{K}_c)$, where $\mathcal{K}_c$ is the critical set with critical value $c$. Algebraic multiplicity of an eigenvalue is introduced in Section \ref{sec:algebra} in order to provide a more precise estimate of the number of independent
eigenvectors. Theorems \ref{th:algebra-nodal} and \ref{th:algebra-span-eigenvector} provide an estimate from below via nodal domains and an equivalent characterization of the algebraic multiplicity, respectively.


As we mentioned, the critical set $\mathcal{K}_c$ could be very huge, in a concrete problem,
it is natural to ask: {\sl How to choose a suitable eigenvector in $K_c$?}
In the last section, from the realistic purpose in finding the Cheeger cut to divide a graph into two parts as equally as possible, we introduce the notion of optimal Cheeger cut via nodal domains in Section \ref{sec:optimal}. We wish this will be helpful in clarifying the results in numerical computations.

\section{Decomposition and Package of eigenvectors}
\label{sec:package}

The eigenvectors of the $1$-Laplacian possess a special property:  In some sense, it can be localized. It makes us possible to develop some special techniques in dealing with the eigenvalue problem for the graph $1$-Laplacian. This section is divided into two parts. In the first part, we observe a {\sl new phenomenon} that any one of the nodal domains of an eigenvector for the $1$-Laplacian is again an eigenvector with the same eigenvalue (see Theorem \ref{th:equivalent-binary}). Accordingly, any eigenvector for the $1$-Laplacian can be decomposed into several fundamental eigenvectors with single nodal domain. In the second part, we introduce the notions of module and eigencomponent and provide some special techniques to put them together, resulting a new eigenvector.

\subsection{Decomposition}

In the following we always assume that eigenvectors and eigenvalues are with respect to $\Delta_1$ defined in Eq.~\eqref{eq:1-Lap-i} on an un-oriented connected graph $G=(V,E)$. The set of all eigenvectors, i.e., all solutions of the system \eqref{eq:eigen-system}, is denoted by $S(G)$.

For a vector $\vec x = (x_1,x_2,\cdots,x_n)\in\mathbb{R}^n$ with $n=|V|$, according to the signatures of $x_i$,
we classify the vertices into three groups:
\begin{equation*}
D^0 = \{i\in V: x_i=0\},
\quad D^\pm = \{i\in V: \pm x_i >0\}.
\end{equation*}
We call $D^0$ the null set of $\vec x$,
and the vertex set of a connected component of the subgraph induced by $D^\pm$ is called a $\pm$ nodal domain. Accordingly, we divide $V$ into $r^+ +r^-$ disjoint $\pm$ nodal domains plus the null set:
\[
V = \big(\bigcup_{\alpha=1}^{r^+}D_\alpha^+\big)\bigcup
\big(\bigcup_{\beta=1}^{r^-}D_\beta^-\big)\bigcup D^0,
\]
where $D^\pm_\gamma$ is a $\pm$ nodal domain and $r^\pm$ is the number of $\pm$ nodal domains.

Assume that an eigenvector $\vec\phi=(x_1, \cdots, x_n)$, has the following nodal domain decomposition
\begin{equation}\label{eq:nodal-decomposition}
 \vec\phi=\big(\sum^{r^+}_{\alpha=1}\sum_{i\in D^+_\alpha}-\sum^{r^-}_{\beta=1}\sum_{i\in D^-_\beta}\big) x_i \vec e_i,
\end{equation}
where $\{\vec e_1,\vec e_2,\cdots,\vec e_n\}$ is the Cartesian basis of $\mathbb{R}^n$. Let $D^\pm_\gamma$ denote any one of $D^+_\alpha$ or $D^-_\beta$, we write $\delta^\pm_\gamma=\vol(D^\pm _\gamma)$, and for any subset $D\subset V$, we define vectors $\vec1_D$ and $\hat{\vec1}_D$ as follow:
\[
(\vec1_D)_i=\left\{\begin{array} {l}
1,\,\,\,\,\,\,\,\,\,\,\,i\in D,\\
0,\,\,\,\,\,\,\,\,\,\,\,i\notin D
\end{array}\right.,
\quad
\hat{\vec1}_D=\frac{\vec1_D}{\vol(D)}.
\]

As we did in \cite{Chang2015} and \cite{ChangShaoZhang2015}, we denote
\begin{equation}\label{eq:pi}
\delta^\pm(\vec x) =\sum_{i\in D^\pm(\vec x)}d_i=\sum_{\gamma=1}^{r^\pm} \delta^\pm_\gamma,
\quad
\delta^0(\vec x) =\sum_{i\in D^0(\vec x)}d_i,
\quad
\pi =\{\vec x\in X: |\delta^+(\vec x)-\delta^-(\vec x)|\le \delta^0(\vec x)\},
\end{equation}
and call that $\vec x$ is equivalent to $\vec y$ in a set $A\subset X$, denoted by $\vec x\simeq \vec y$, if there exists a path $l(t)$ in $X$ such that $l(0)=\vec x$, $l(1)=\vec y$ and $l(t)\in A$ for any $t\in[0,1]$.

\begin{theorem}
\label{th:equivalent-binary}
If $\vec x\in S(G)$ has the nodal domain decomposition \eqref{eq:nodal-decomposition},
then $\forall\, \alpha\in\{1,\cdots,r^+\}$ (or $\forall\, \beta\in\{1,\cdots,r^-\}$)
$\vec x\simeq \hat{\vec 1}_{D^+_\alpha}$ (or $\hat{\vec 1}_{D^-_\beta}$) in $S(G)\cap I^{-1}(\mu)$ with
$\mu=I(\vec x)$.
\end{theorem}

\begin{proof}
According to Theorem 3.8 in \cite{Chang2015}, we may assume
$$
\vec x=\vec \xi=r(\vec E^+-\vec E^-),
\quad
r=\frac{1}{\delta^++\delta^-},
\quad
\vec E^\pm=(\sum^{r^\pm}_{\gamma=1}\sum_{i\in D^\pm_\gamma})\vec e_i,
$$
and then we have
\begin{equation}\label{eq:xi}
\mu = I(\vec \xi),
\quad
\sum_{j\sim i} z_{ij}(\vec \xi)\in \mu \sgn(\xi_i),\,\,\,i=1,2,\cdots, n.
\end{equation}
Now for any $t\in[0,1]$ let
\[
\vec \xi_t =\rho_t\sum_{i\in D^+_\alpha}\vec e_i+rt\sum_{\gamma\neq \alpha}\sum_{i\in D^+_\gamma}\vec e_i-rt \vec E^-=(\rho_t-rt)\sum_{i\in D^+_\alpha}\vec e_i+rt (\vec E^+-\vec E^-),
\]
where
$$
(\rho_t-rt)\delta^+_\alpha +t=1.
$$
It can be readily checked that $\vec \xi_t$ satisfies the same system \eqref{eq:xi} for all $t\in (0, 1]$,
which implies that $\vec \xi_0\in S(G)$ and $I(\vec \xi_0)=\lim_{t\to 0+}I(\vec \xi_t)=\mu$ because
$S(G)$ is closed.
However,
\[
\vec \xi_0=\rho_0 \vec 1_{D^+_\alpha}= \hat{\vec 1}_{D^+_\alpha}.
\]
The proof is thus completed.
\end{proof}

\begin{cor}
\label{cor:nodal-eigenvector}
Let $(\mu,\vec\phi)$ be an eigenpair of $\Delta_1$. If $\vec \phi$ has $k$ nodal domains $\{D_\gamma \,|\, 1\le \gamma \le k\}$, then $\{\hat{\vec1}_{D_\gamma} \,|\, 1\le \gamma \le k\}$ are all eigenvectors with the same eigenvalue $\mu$.
\end{cor}

\begin{cor}\label{cor:delta+<delta}
Under the assumption of Theorem \ref{th:equivalent-binary}, we have
$$
\vol(D^\pm _\gamma) \le \frac{1}{2}\vol(V),
$$
where $D^\pm_\gamma$ denotes any one of $D^+_\alpha$ or $D^-_\beta$.
\end{cor}

\begin{proof} Take $D^+_\alpha$ as an example. Since $\vec x$ is an eigenvector, according to Theorem 3.11 in \cite{Chang2015}, we have $|\delta^+-\delta^-|\le \delta^0,$ which implies
$\delta^+\le \delta^- + \delta^0$. But $\delta^+ + \delta^- + \delta^0=\sum^n_{i=1} d_i$, therefore
$$
\vol (D_\alpha^+) = \delta^+_\alpha\le \delta^+\le \frac{1}{2}\sum^n_{i=1} d_i = \frac{1}{2}\vol(V).
$$
The proof is finished.
\end{proof}

The new finding presented in Theorem \ref{th:equivalent-binary}
serves as not only the jumping-off place but also the core foundation of the present study,
which reflects in at least three aspects below.
\begin{enumerate}[1.]
\item It provides a systematic way of computing all eigenvalues of the graph $1$-Laplacian by investigating only the vectors with single nodal domain. In this view, we obtain the entire spectrum of the complete graph $K_n$ (see Section \ref{sec:Kn}) and all the eigenvalues of $1$-Laplacian for the graph in Example \ref{fig:counterexample} (see Section \ref{sec:example}).

\item Besides the topological multiplicity (via genus) of an eigenvalue, the notion of algebraic multiplicity of an eigenvalue is derived from Theorem \ref{th:equivalent-binary} (see Section \ref{sec:algebra}). In particular, there may be many Cheeger cuts (i.e., binary eigenvectors corresponding to the second eigenvalue), and this leads us to introduce the optimal Cheeger cut (see Section \ref{sec:optimal}).

\item The localization property of eigenvectors gives us an opportunity to combine the eigenvectors with disjoint supports to form a new module for a larger graph. This guides us to consider the package (see Section \ref{subsec:package}).
Further with the aid of several special techniques, we successfully compute the eigenvectors with the most possible nodal domains of path graphs $P_n$ as well as cycle graphs $C_n$ (see Section \ref{sec:PnCnKn}).
\end{enumerate}

\subsection{Package}
\label{subsec:package}

Conversely, we can put together some pieces of small graphs with the same prescribing eigenvalue $\mu$. Let us first introduce the notions of a piece of graph, which will be used as a component or a module of a graph.

\begin{defn}[$\mu$-module]
\label{def:module}
Let $H=(V_H, E_H)$ be a connected graph and $\vec\phi\in X$. Assume
\begin{enumerate}[(1)]
\item $V_H=V_H^o \cup V_H^s$, satisfying $V_H^o\cap V_H^s=\emptyset$, and there are no inter-connections among vertices in $V_H^s$;
\item $\sum_{j\sim i}z_{ij}(\vec\phi)\in \mu d_i \sgn(\phi_i),\,\,\forall\, i\in V_H^o.$
\end{enumerate}
Then $(H, \vec\phi)$ is called a $\mu$-{\sl module}.
Here $V_H^o$ is called the {\sl core}, and $V_H^s$ the {\sl socket}.
\end{defn}

In the definition, $V_H^s$ could be an empty set.

\begin{defn}[$\mu$-eigencomponent]
\label{def:eigen-component}
A $\mu$-module $(H, \vec\phi)$ is called a $\mu$-{\sl eigencomponent}, if $V_H^o$ is connected, and $\vec\phi=\hat{\vec1}_{A}$ for some nonempty set $A\subset V_H^o$.
\end{defn}

We have some simple propositions as follow.

\begin{pro}
\label{pro:mu-eigencomponent}
Let $\vec \phi$ be a $\mu$ eigenvector of the graph $G=(V, E)$, and $D\subset V$ be a nodal domain of $\vec\phi$. If there exists a vertex $a\in V\backslash D$, which is adjacent to $D$, then $(H, \hat{\vec1}_D)$ is a $\mu$-eigencomponent, where $H=(V_H, E|_{V_H})$, $V_H=D\cup \{a\}$, $V_H^o=D$,
and $V_H^s=\{a\}$.
\end{pro}

\begin{proof}
Since $a$ is adjacent to $D$, $V_H$ is connected, and $V_H^s=\{a\}$ is a single vertex set, there are no inter connections in $V_H^s$.
Moreover, $D$ is a nodal domain of $\vec\phi$ with eigenvalue $\mu$, the system in (2) of Definition \ref{def:module} holds automatically.
\end{proof}

\begin{pro}
\label{pro:mu-eigenvector}
If $(H, \vec\phi)$ is a $\mu$-eigencomponent with $V_H^s=\emptyset$, then $\vec\phi$ is an eigenvector of the graph $H$ with eigenvalue $\mu$.
\end{pro}

\begin{proof}
It follows from the definition directly.
\end{proof}

\begin{pro}
If $(\mu, \vec\phi)$ is an eigenpair of $1$-Laplacian on a connected graph $G$, and $\vec\phi$ has the following nodal domain decomposition:
$$
\vec\phi=\frac{1}{\delta^++\delta^-}(\sum_{\alpha=1}^{r^+} \vec 1_{D^+_\alpha}-\sum_{\beta=1}^{r^-} \vec 1_{D^-_\beta}),
$$
then $\forall\,\gamma\in\{1,\cdots,r^\pm\}$, $(H_{D^\pm_\gamma}, \hat{\vec 1}_{D^\pm_\gamma})$ is a $\mu$-eigencomponent, where
$V_{H_{D^\pm_\gamma}}^o=D^\pm_\gamma$,
and $ V_{H_{D^\pm_\gamma}}^s=\emptyset$.
\end{pro}

\begin{proof}
It is a combination of Theorem \ref{th:equivalent-binary} and Proposition \ref{pro:mu-eigenvector}.
\end{proof}

Next we provide some examples on modules and eigencomponents.

\begin{example}\rm
\label{exam:P3}
Fig.~\ref{fig:P3} shows a $\frac13$-eigencomponent with only one socket denoted by $\mathcal{M}_1$ hereafter. Here $\mu=\frac13$, $V=\{1,2,3\}$, $E=\{(12), (23)\}$,
$\vec \phi=\frac13 (1,1,0)$, $V^o=\{1,2\}$, $V^s=\{3\}$, $z_{12}=\frac13$.
\begin{figure}[htpb]
\centering
\begin{tikzpicture}[auto]
\node (1) at (0,0) {$\bullet$};
\node (2) at (1,0) {$\bullet$};
\node (3) at (2,0) {$0$};
\node (5) at (3.5,0) {$=$};
\node (4) at (5,0) {$P_2\sim\circ$};
\node (1-) at (0.3,0.3) {$1$};
\node (2-) at (1.3,0.3) {$2$};
\node (3-) at (2.3,0.3) {$3$};
\draw (1) to (2);
\draw (2) to (3);
\draw (1) circle(0.23);
\draw (2) circle(0.23);
\draw (3) circle(0.23);
\end{tikzpicture}
\caption{A $\frac13$-eigencomponent with one socket used in Example \ref{exam:P3}. Hereafter, the components of $\vec\phi$ corresponding to the vertices marked by bullets are either all positive or all negative. }
\label{fig:P3}
\end{figure}
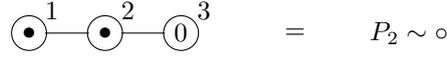
\end{example}

\begin{example}\rm
\label{exam:C3}
A $\frac12$-eigencomponent with only one socket, denoted by $\mathcal{M}_2$ hereafter,  is shown in Fig.~\ref{fig:C3}. We take $\mu=\frac12$,
$V=\{1,2,3\}$, $E=\{(12), (13), (23)\}$,
$\vec\phi=\frac14 (1,1,0)$, $V^o=\{1,2\}$, $V^s=\{3\}$, $z_{12}=0$.
\begin{figure}[htpb]
\centering
\begin{tikzpicture}[auto]
\node (1) at (0,0) {$\bullet$};
\node (2) at (2,0) {$\bullet$};
\node (3) at (1,1) {$0$};
\node (1-) at (0,0.4) {$1$};
\node (2-) at (2,0.4) {$2$};
\node (3-) at (1.3,1.3) {$3$};
\draw (1) to (2);
\draw (2) to (3);
\draw (1) to (3);
\draw (1) circle(0.23);
\draw (2) circle(0.23);
\draw (3) circle(0.23);
\end{tikzpicture}
\caption{A $\frac12$-eigencomponent with one socket used in Example \ref{exam:C3}.}
\label{fig:C3}
\end{figure}
\end{example}

\begin{example}\rm
\label{exam:6v}
A $\frac17$-eigencomponent with two sockets,
denoted by $\mathcal{M}_3$ hereafter, is displayed in
Fig.~\ref{fig:6v}. Here $\mu=\frac17$,
$V=\{1,2,3,4,5,6\}$,
$E=\{(12), (13), (14), (15), (23), (24), (34), (36)\}$,
$\vec\phi=\frac{1}{14}(1,1,1,1,0,0)$,
$V^o=\{1,2,3,4\}$,
$V^s=\{5,6\}$,
$z_{13}=z_{24}=0$,
$z_{12}=z_{14}=z_{32}=z_{34}=\frac{3}{14}$.
\begin{figure}[htpb]
\centering
\begin{tikzpicture}[auto]
\node (3) at (0,0) {$\bullet$};
\node (4) at (-2,2) {$\bullet$};
\node (2) at (2,2) {$\bullet$};
\node (1) at (0,4) {$\bullet$};
\node (5) at (-2,4) {$0$};
\node (6) at (2,0) {$0$};
\node (8) at (3.3,2) {$=$};
\node (7) at (5,2) {$\circ\sim K_4\sim \circ$};
\node (3-) at (-0.35,0) {$3$};
\node (4-) at (-2.3,2.3) {$2$};
\node (2-) at (2.3,2.3) {$4$};
\node (1-) at (0.3,4.3) {$1$};
\node (5-) at (-2.3,4.3) {$5$};
\node (6-) at (2.3,0.3) {$6$};
\draw (4) to (3);
\draw (4) to (2);
\draw (3) to (2);
\draw (1) to (2);
\draw (1) to (3);
\draw (1) to (4);
\draw (3) to (6);
\draw (1) to (5);
\draw (1) circle(0.23);
\draw (2) circle(0.23);
\draw (3) circle(0.23);
\draw (4) circle(0.23);
\draw (5) circle(0.23);
\draw (6) circle(0.23);
\end{tikzpicture}
\caption{A $\frac17$-eigencomponent with two sockets used in Example \ref{exam:6v}.}
\label{fig:6v}
\end{figure}
\end{example}

\begin{example}\rm
\label{exam:P5}
A $\frac13 $-eigencomponent with two sockets, denoted by $\mathcal{M}_4$ hereafter, is given in Fig.~\ref{fig:P5}. Here we take
$\mu=\frac13$,
$V=\{1,2,3,4,5\}$,
$E=\{(12), (23), (34), (45)\}$,
$\vec\phi=\frac16 (0,1,1,1,0)$,
$V^o=\{2,3,4\}$,
$V^s=\{1,5\}$,
$z_{23}=-\frac13$,
$z_{34}=\frac13$.
\begin{figure}[htpb]
\centering
\begin{tikzpicture}[auto]
\node (1) at (1,0) {$0$};
\node (2) at (2,0) {$\bullet$};
\node (3) at (3,0) {$\bullet$};
\node (4) at (4,0) {$\bullet$};
\node (5) at (5,0) {$0$};
\node (8) at (6.3,0) {$=$};
\node (6) at (8,0) {$\circ\sim P_3\sim \circ$};
\node (1-) at (1.3,0.3) {$1$};
\node (2-) at (2.3,0.3) {$2$};
\node (3-) at (3.3,0.3) {$3$};
\node (4-) at (4.3,0.3) {$4$};
\node (5-) at (5.3,0.3) {$5$};
\draw (1) to (2);
\draw (2) to (3);
\draw (3) to (4);
\draw (4) to (5);
\draw (1) circle(0.23);
\draw (2) circle(0.23);
\draw (3) circle(0.23);
\draw (4) circle(0.23);
\draw (5) circle(0.23);
\end{tikzpicture}
\caption{A $\frac13$-eigencomponent with two sockets used in Example \ref{exam:P5}.}
\label{fig:P5}
\end{figure}
\end{example}

\begin{example}\rm
\label{exam:P6}
A $\frac13$-module with one socket is shown in Fig.~\ref{fig:P6}. We take here
$\mu=\frac13$,
$V=\{1,2,3,4,5,6\}$,
$E=\{(12), (23), (34), (45), (56)\}$,
$\vec\phi=\frac19 (1,1,-1,-1,-1,0)$,
$V^o=\{1,2,3,4,5\}$,
$V^s=\{6\}$,
$z_{12}=\frac13$,
$z_{34}=-\frac13$,
$z_{45}=\frac13$.
This module can be seen as a plugging module of $\mathcal{M}_1$ and $\mathcal{M}_4$ (see Section \ref{subsec:plugging}).
\begin{figure}[htpb]
\centering
\begin{tikzpicture}[auto]
\node (1) at (1,0) {$+$};
\node (2) at (2,0) {$+$};
\node (3) at (3,0) {$-$};
\node (4) at (4,0) {$-$};
\node (5) at (5,0) {$-$};
\node (6) at (6,0) {$0$};
\node (8) at (7.3,0) {$=$};
\node (7) at (9,0) {$\M_1\plug\M_4$};
\node (1-) at (1.3,0.3) {$1$};
\node (2-) at (2.3,0.3) {$2$};
\node (3-) at (3.3,0.3) {$3$};
\node (4-) at (4.3,0.3) {$4$};
\node (5-) at (5.3,0.3) {$5$};
\node (6-) at (6.3,0.3) {$6$};
\draw (1) to (2);
\draw (2) to (3);
\draw (3) to (4);
\draw (4) to (5);
\draw (5) to (6);
\draw (1) circle(0.23);
\draw (2) circle(0.23);
\draw (3) circle(0.23);
\draw (4) circle(0.23);
\draw (5) circle(0.23);
\draw (6) circle(0.23);
\end{tikzpicture}
\caption{A $\frac13$-module used in Example \ref{exam:P6}.}
\label{fig:P6}
\end{figure}
\end{example}

All the eigencomponents shown in Examples \ref{exam:P3}, \ref{exam:C3}, \ref{exam:6v} and \ref{exam:P5} are of course modules, but
the module given in Example \ref{exam:P6} is not an eigencomponent. In the sequel of this section, we shall introduce a few special techniques to construct a new $\mu$-module from given $\mu$-modules.

\subsection{Extension}

Given a $\mu$-module $(H, \vec\phi)$ of a graph $G$, sometimes we can simply extend $\vec\phi$ to $\tilde{\vec\phi}$ by $0$ outside $H$, such that $\tilde{\vec\phi}$ is an eigenvector of $G$.

\begin{pro}
\label{pro:extension}
Let $H$ be a subgraph of $G=(V, E)$, and let $(H, \vec\phi)$ be a $\mu$-eigencomponent. Assume $\vec\phi$ is an eigenvector on $H$ with $\vec\phi|_{V_H^s}=0$, together with
\begin{equation}
\label{eq:extension-condition}
 j\sim V_H\,\mbox{and}\,j\notin V_H \,\Rightarrow\, j \,\mbox{is disconnected with}\, V_H^o.
\end{equation}
Then
\begin{equation*}
\tilde{\phi}_i=
\left\{
\begin{array}{l}
\phi_i,\,\,\,\,\,i\in V_H\\
0,\,\,\,\,\,\,\,i\notin V_H,
\end{array}
\right.
\end{equation*}
is an eigenvector with eigenvalue $\mu$ on $G$.
\end{pro}

\begin{proof}
We verify the system \eqref{eq:eigen-system} for eigenvectors.  For $i\in V_H^o$, by the assumption \eqref{eq:extension-condition}: if $j\sim i$, then $j\in V_H.$ Since $(H, \vec\phi)$ is an eigencomponent,
\begin{equation*}
\sum_{j\sim i,~j\in V_H} z_{ij}(\tilde{\vec\phi})=\sum_{j\sim i,~j\in V_H} z_{ij}(\vec\phi)\in \mu d_i \sgn(\phi_i)=\mu d_i \sgn(\tilde{\phi}_i).
\end{equation*}
As to $i\notin V_H$, if $j\sim i$, then either $j\notin V_H$ or $j\in V_H^s$ again by the assumption \eqref{eq:extension-condition}, in both cases, $\tilde{\phi}_j=0$. Since $\tilde{\phi}_i=0$, one may choose $z_{ij}(\tilde{\vec\phi})=0$. Thus
\begin{equation*}
\sum_{j\sim i} z_{ij}(\tilde{\vec\phi})=0\in \mu d_i [-1, 1]=\mu d_i \sgn(\tilde{\phi}_i).
\end{equation*}
It remains to consider the case $i\in V_H^s$. Since now $\tilde{\phi}_i= \phi_i=0,$ and
\begin{equation*}
\sum_{j\sim i} z_{ij}(\tilde{\vec\phi})=(\sum_{j\sim i, j\in V_H}+  \sum_{j\sim i, j\notin V_H})  z_{ij}(\tilde{\vec\phi}).
\end{equation*}
Again we may set $z_{ij}(\tilde{\vec\phi})=0,\,\forall\, j\notin V_H,$ we obtain
\begin{equation*}
\sum_{j\sim i} z_{ij}(\tilde{\vec\phi})=\sum_{j\sim i} z_{ij}(\vec\phi) \in \mu \tilde{d} \sgn(\tilde{\phi}_i),
\end{equation*}
where $\tilde{d}$ is the degree of $i$ in $V_H$ provided that $\vec\phi$ is an eigenvector on $H$.  However, $\tilde{d}\le d_i$, the degree of $i$ in $V$, it follows
\begin{equation*}
\sum_{j\sim i} z_{ij}(\tilde{\vec\phi})=\sum_{j\sim i} z_{ij}(\vec\phi) \in \mu d_i \sgn(\tilde{\phi}_i).
\end{equation*}
The proof is completed.
\end{proof}

As applications of Proposition \ref{pro:extension}, we embed the modules, for example,
\begin{enumerate}
\item $H=\{1,2,3,4\}$, $\vec\phi=\frac13 (1,1,0,0)$, $\mu=\frac13$ to $P_n, n>4$;
\item $H=\{1,2,3,4,5,6\}$,
$\vec\phi=\frac14 (0,0,1,1,0,0)$,
$\mu=\frac12$ to $C_n, n>6$.
\end{enumerate}
Then eigenvectors for these two larger graphs are obtained by the extension.

\subsection{Joining}
\label{subsec:join}

\begin{pro}
Let $\{(H_i, \vec\phi_i)\,|\, i=1, \cdots, m\}$ be $\mu$-modules with $V_i=V^o_i\cup V^s_i, \,i=1, \cdots, m$. Assume
\begin{enumerate}[(1)]
\item $\exists\, u_i\in V^s_i$ with $(\vec\phi_i)_{u_i}=0$;
\item $\exists$ an edge set $E'$ to the vertex set $V'=\{u_1, \cdots, u_m\}$ together with $\alpha_{ij}\in [-1,1]$, satisfying
\begin{align*}
\alpha_{ij}&=-\alpha_{ji},\,\,\forall\, (ij)\in E', \\
\sum_{v\in V_i, v\sim u_i}z_{u_j u_i}&+\sum_{u_j\sim u_i}\alpha_{ij}\in \mu(d_{u_i}+d'_{u_i})[-1, 1],\,\,\forall u_i\in V'
\end{align*}
where $d'_{u_i}$ is the degree of $u_i$ in $V'$.
\end{enumerate}
Let
$$
V=\cup^m_{i=1} V_i,\quad
E=(\cup^m_{i=1}E_i) \cup E',\quad
\vec\phi=\oplus^m_{i=1}\vec\phi_i.
$$
Then $(H,\vec\phi)$ with $H=(V, E)$ is also a $\mu$-module, with
$$
V^o=(\cup^m_{i=1} V^o_i)\cup V', \quad
V^s=\cup^m_{i=1}( V_i^s\backslash\{u_i\}).
$$
In this case, we call $(H, \vec\phi)$ a joining module of modules $\{(H_i, \vec\phi_i)\}$,
denoted by
$$
(H, \vec\phi) = (H_1, \vec\phi_1)\join(H_2, \vec\phi_2)\cdots\join(H_m, \vec\phi_m).
$$
\end{pro}

The above proposition can be readily verified by straightforward calculations.

\begin{example}\rm
\label{exam:order9G}
 We join three copies of $\mathcal{M}_2$ (see Example \ref{exam:C3}), and then produce a new module with eigenvalue $\frac12$. Here we set $\vec\phi_1=\vec\phi_2=-\vec\phi_3=\vec\phi$, $E'=\{(u_1 u_2), (u_2 u_3), (u_3 u_1)\}$, $\alpha_{u_1u_2}=\alpha_{u_2u_3}= \alpha_{u_3u_1}=0$,
where $u_i\in D^0(\vec\phi_i)$. Accordingly, we obtain a joining module with eigenvalue $\frac12 $ and eigenvector $\frac{1}{12}(0,0,0,1,1,1,1,-1,-1)$ on $9$ vertices displayed in Fig.~\ref{fig:order9G}. 
In fact, $V_1=\{4, 5, u_1\}, V_2=\{6,7, u_2\}, V_3=\{8,9, u_3\}$;  $V^o_1=\{4,5\}, V^o_2=\{6,7\}, V^o_3=\{8,9\}$; $V^s_i=\{u_i\}, i=1,2,3$. Thus $V^o=\{4,5,6,7,8,9, u_1,u_2,u_3\}, V^s=\emptyset$. Moreover, let $V'_i=V_i\cup\{v_i\}, i=1,2,3$, $E_1=\{(4,5), (u_1, 4), (u_1, 5), (u_1, v_1)\}$, similarly for $E_2, E_3$,  and  $V'^o_i=V^o_i, V^s=\{u_i, v_i\,|\, i=1,2,3\}$, where $v_1,v_2,v_3$ are the new added vertices and they are not displayed in Fig.~\ref{fig:order9G}. Then
$V'^o=V^o, V'^s=\{v_1, v_2, v_3\}$.
Again, we obtain a $\frac12 $-eigencomponent.
\begin{figure}[htpb]
\centering
\begin{tikzpicture}[auto]
\node (1) at (0,1) {$0$};
\node (2) at (-1,-1) {$0$};
\node (3) at (1,-1) {$0$};
\node (4) at (-1,3) {$+$};
\node (5) at (1,3) {$+$};
\node (6) at (3,-1) {$+$};
\node (7) at (2,-3) {$+$};
\node (8) at (-2,-3) {$-$};
\node (9) at (-3, -1) {$-$};
\node (1-) at (0,1.4) {$u_1$};
\node (2-) at (-1.3,-0.7) {$u_3$};
\node (3-) at (1.3,-0.7) {$u_2$};
\node (4-) at (-1.3,3.3) {$4$};
\node (5-) at (1.3,3.3) {$5$};
\node (6-) at (3.3,-1.3) {$6$};
\node (7-) at (2.3,-3.3) {$7$};
\node (8-) at (-2.3,-3.3) {$8$};
\node (9-) at (-3.3, -1.3) {$9$};
\draw (1) to (2);
\draw (2) to (3);
\draw (3) to (1);
\draw (1) to (4);
\draw (1) to (5);
\draw (2) to (8);
\draw (2) to (9);
\draw (3) to (6);
\draw (3) to (7);
\draw (4) to (5);
\draw (6) to (7);
\draw (8) to (9);
\draw (1) circle(0.23);
\draw (2) circle(0.23);
\draw (3) circle(0.23);
\draw (4) circle(0.23);
\draw (5) circle(0.23);
\draw (6) circle(0.23);
\draw (7) circle(0.23);
\draw (8) circle(0.23);
\draw (9) circle(0.23);
\end{tikzpicture}
\caption{A joining module used in Example \ref{exam:order9G}.}
\label{fig:order9G}
\end{figure}
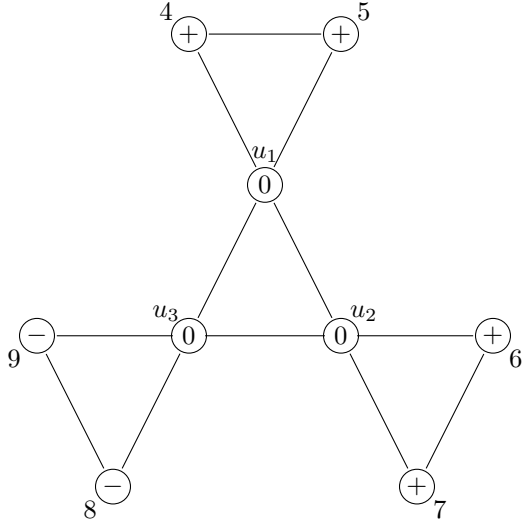
\end{example}

\subsection{Pasting}

\begin{pro}
\label{pro:paste}
Let $\{(H_i, \vec\phi_i)\,|\,1\le i\le m\}$ be $\mu$-modules and $\vec\phi_i=\sum_{v\in V_i} x^i_v \vec e^i_v, \,1\le i\le m$. Suppose there exists $\, u_i\in V_i^s$ such that $(\vec\phi_i)_{u_i},\,1\le i\le m$ are all equal, and
\begin{equation}
\label{assump:u}
\sum_{i=1}^m \sum_{v\sim u_i} z_{u_i v}(\vec\phi_i)\in \mu \sum^m_{i=1}d_{u_i}\sgn((\vec\phi_i)_{u_i}).
\end{equation}
We identify the $m$ vertices $\{u_i\},1\le i\le m$, as one vertex $\{u\}$, and let $V=\cup^m_{i=1} (v_i\backslash \{u_i\})\sqcup \{u\},\, E=\cup^m_{i=1} E_i$, in which edges of the form $(v u_i)$ are identified with $(v u)$, and $H=(V, E)$.
Then $(H, \vec\phi)$ is a $\mu$-module, where
$$
V^o=\cup^m_{i=1} V_i^o \cup \{u\},\,\, \,
V^s=\cup^m_{i=1} (V_i^s\backslash\{u_i\}), \,\,\, \vec\phi=\sum^m_{i=1}\sum_{v\neq u_i}x^i_v \vec e^i_v+\phi_u \vec e_u,
$$
and $\phi_u$ is defined to be $(\vec\phi_i)_{u_i}, \,\forall\,i\in\{1,\cdots,m\}$ as assumed.
In this case, we call $(H, \vec\phi)$ a pasting module of modules $\{(H_i, \vec\phi_i)\}$,
denoted by
$$
(H, \vec\phi) = (H_1, \vec\phi_1)\paste(H_2, \vec\phi_2)\cdots\paste(H_m, \vec\phi_m).
$$
\end{pro}

The verification of Proposition \ref{pro:paste} is as follows. We only need to verify the vertices in $V^o\backslash\{u\}$ provided by the assumption \eqref{assump:u}. Actually, $\forall\, v\in V^o\backslash\{u\}$ there exists unique $j$ such that $v\in V^o_j$, and then all vertices adjacent to $v$ must be in $V_j$. Accordingly, we have
$$
\sum_{w\sim v, w\in V}z_{wv}(\vec\phi)=\sum_{w\sim v,w\in V_j}z_{wv}\in \mu d_v \sgn((\vec\phi_j)_v)=\mu d_v \sgn(\phi_v).
$$
Therefore the verification is completed.

\begin{example}\rm
\label{exam:P7}
We paste the two modules $\mathcal{M}_1=P_2\sim\circ$ (see Example \ref{exam:P3}) and $\mathcal{M}_4=\circ\sim P_3\sim \circ$ (see Example \ref{exam:P5})  by identifying the two vertices, $u_1=\{3\}$ in Example \ref{exam:P3} and $u_2=\{1\}$ in Example \ref{exam:P5} as a new vertex $u=\{3\}$. The resulting module is shown in Fig.~\ref{fig:P7}.
Now we have
$\mu=\frac13$,
$V=\{1,2,3,4,5,6,7\}$,
$E=\{(12), (23), (34), (45), (56), (67)\}$, $\vec\phi=\frac19 (1,1,0,-1,-1,-1,0)$,
$V^o=\{1,2,3,4,5,6\}$,
$V^s=\{7\}$.
\begin{figure}[htpb]
\centering
\begin{tikzpicture}[auto]
\node (1) at (1,0) {$+$};
\node (2) at (2,0) {$+$};
\node (3) at (3,0) {$0$};
\node (4) at (4,0) {$-$};
\node (5) at (5,0) {$-$};
\node (6) at (6,0) {$-$};
\node (7) at (7,0) {$0$};
\node (1-) at (1.3,0.3) {$1$};
\node (2-) at (2.3,0.3) {$2$};
\node (3-) at (3.3,0.3) {$3$};
\node (4-) at (4.3,0.3) {$4$};
\node (5-) at (5.3,0.3) {$5$};
\node (6-) at (6.3,0.3) {$6$};
\node (7-) at (7.3,0.3) {$7$};
\draw (1) to (2);
\draw (2) to (3);
\draw (3) to (4);
\draw (4) to (5);
\draw (5) to (6);
\draw (6) to (7);
\draw (1) circle(0.23);
\draw (2) circle(0.23);
\draw (3) circle(0.23);
\draw (4) circle(0.23);
\draw (5) circle(0.23);
\draw (6) circle(0.23);
\draw (7) circle(0.23);
\end{tikzpicture}
\caption{A pasting module used in Example \ref{exam:P7}.}
\label{fig:P7}
\end{figure}
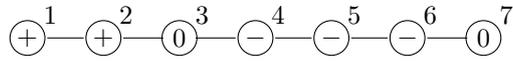
\end{example}

\begin{example}\rm
\label{exam:5orderG}
One may paste up two copies of the modules in Example \ref{exam:C3} by identifying the null vertices in both copies and by changing the sign of the function in one of the two copies. The pasting module of two $\M_2$ shown in Fig.~\ref{fig:5orderG}, where we set
$\mu=\frac12$,
$V=\{1,2,3,4,5\}$,
$E=\{(12), (15), (25), (35), (45), (34)\}$, $\vec\phi=\frac18 (1,1,-1,-1,0)$,
$V^o=\{1,2,3,4\}$,
$V^s=\{5\}$. In fact, it can be easily checked that $(\frac12 , \vec\phi)$ is indeed an eigenpair of the graph $(V, E)$.
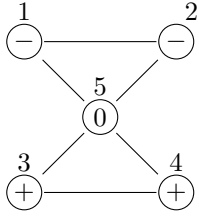
\begin{figure}[htpb]
\centering
\begin{tikzpicture}[auto]
\node (1) at (-1,1) {$-$};
\node (2) at (1,1) {$-$};
\node (3) at (-1,-1) {$+$};
\node (4) at (1,-1) {$+$};
\node (5) at (0,0) {$0$};
\node (1-) at (-1,1.4) {$1$};
\node (2-) at (1.2,1.4) {$2$};
\node (3-) at (-1,-0.6) {$3$};
\node (4-) at (1,-0.6) {$4$};
\node (5-) at (0,0.4) {$5$};
\draw (1) to (2);
\draw (1) to (5);
\draw (2) to (5);
\draw (3) to (4);
\draw (3) to (5);
\draw (4) to (5);
\draw (1) circle(0.23);
\draw (2) circle(0.23);
\draw (3) circle(0.23);
\draw (4) circle(0.23);
\draw (5) circle(0.23);
\end{tikzpicture}
\caption{A pasting module used in Example \ref{exam:5orderG}.}
\label{fig:5orderG}
\end{figure}
\end{example}

\subsection{Plugging}
\label{subsec:plugging}

\begin{pro}
\label{pro:plug}
Let $\{(H_i, \vec\phi_i)\,|\,i=1,2\}$ be two $\mu$-components. Assume
\begin{enumerate}[(1)]
\item $\exists\, \tilde{V}^s_i=\{u^i_1,\cdots, u^i_m\}\subset V_i^s,\,i=1,2$, and $\forall\,j,\,\exists$ unique $v^i_j\in V^o_i$ such that $v^i_j$ is uniquely adjacent to $u^i_j, 1\le j\le m,\, i=1,2,$ and $(\vec\phi_i)_{v_j^i}=1$;
\item $\vec\phi_i|_{\tilde{V}^s_i}=0,\;\forall\, i\in\{1,2\}$.
\end{enumerate}
Let
$$
V=(V_1\backslash \tilde{V_1}^s)\cup (V_2\backslash \tilde{V_2}^s), \,\,\,
E=(E_1|_{V_1\backslash\tilde{V_1^s}})\cup (E_2|_{V_2\backslash\tilde{V_2^s}})\cup (\cup^m_{j=1} \{(v^1_j v^2_j)\}),
\,\,\,
H=(V,E),
$$
where both $(v^1_j u^1_j)$ and $(v^2_j u^2_j)$ identify with $(v^1_j v^2_j)$, and
$$
\vec\phi=\vec\phi_1|_{V_1\backslash\tilde{V}_1^s}\oplus (-\vec\phi_2)|_{V_2\backslash\tilde{V}_2^s}.
$$
Then $(H, \vec\phi)$ is a $\mu$-module, where
$$
V^o=V^o_1\cup V^o_2, \,\,\,
V^s=(V^s_1\backslash \tilde{V}^s_1)\cup (V^s_2\backslash \tilde{V}^s_2).
$$
In this case, we call $(H, \vec\phi)$ a plugging module of $(H_1, \vec\phi_1)$ and $(H_2, \vec\phi_2)$, denoted by
\[
(H, \vec\phi) = (H_1, \vec\phi_1)\plug (H_2, \vec\phi_2).
\]
\end{pro}

The verification of Proposition \ref{pro:plug} is as follows. To vertices other than $\{v^1_j, v^2_j:\,1\le j\le m\}$, the equations in system \eqref{eq:eigen-system} are trivially true. We only verify those equations at
$\{v^1_j, v^2_j:\,1\le j\le m\}$. For $\vec\phi_1$ and $\vec\phi_2$, it can be easily checked that
\begin{align*}
\sum_{k\sim v^1_j, k\in V_1\backslash\tilde{V}_1^s} z_{v^1_j k}(\vec\phi_1)+z_{v^1_j u^1_j}(\vec\phi_1)&=\mu d_{v^1_j},\,\,\forall\, 1\le j\le m, \\
\sum_{k\sim v^2_j, k\in V_2\backslash\tilde{V}_2^s} z_{v^1_j k}(-\vec\phi_2)+z_{v^2_j u^2_j}(-\vec\phi_2) & =-\mu d_{v^2_j},\,\,\forall\, 1\le j\le m, \\
z_{v^1_j u^1_j}(\vec\phi_1)=1, \,\,\, z_{v^2_j u^2_j}(-\vec\phi_2)&=-1.
\end{align*}
In consequence, if $k\sim v^1_j$ and $k\neq u^1_j$, then we have $k\in V_1\backslash\tilde{V}_1^s$,
and
\begin{align*}
\sum_{k\sim v^1_j, k\neq u^1_j} z_{v^1_j k}(\vec\phi)+z_{v^1_j v^2_j}(\vec\phi)=\mu d_{v^1_j}&=\mu d_{v^1_j}\sign((\vec\phi)_{v^1_j}),\,\,\forall\, 1\le j\le m.
\end{align*}
Similarly, 
\begin{align*}
\sum_{k\sim v^2_j, k\neq u^2_j} z_{v^2_j k}(\vec\phi)+z_{v^2_j v^1_j}(\vec\phi)=-\mu d_{v^2_j}&=\mu d_{v^2_j}\sign((\vec\phi)_{v^2_j}),\,\,\forall\, 1\le j\le m.
\end{align*}
Therefore the verification is finished.

A plugging module, $\mathcal{M}_1\plug\mathcal{M}_4$,  has been shown in Example \ref{exam:P6}. Below we present another example.

%

\begin{example}\rm
\label{exam:P5plugtoC6}
 We plug two copies of the module $\mathcal{M}_4=\circ \sim P_3\sim\circ$ (see Example \ref{exam:P5}), first by identifying the null vertices in one module to their copies in the other, after eliminating them, we change the sign of the function $\vec\phi$ in one of the two copies. The resulting
module is $\mathcal{M}_4\plug\mathcal{M}_4$ which is shown in Fig.~\ref{fig:P5plugtoC6},  with $\mu=\frac13$, $\vec\phi=\frac{1}{12}(1,1,1,-1,-1,-1)$.
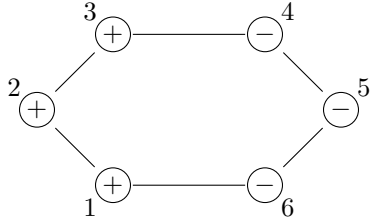
\begin{figure}[htpb]
\centering
\begin{tikzpicture}[auto]
\node (1) at (-1,-1) {$+$};
\node (2) at (-2,0) {$+$};
\node (3) at (-1,1) {$+$};
\node (4) at (1,1) {$-$};
\node (5) at (2,0) {$-$};
\node (6) at (1,-1) {$-$};
\node (1-) at (-1.3,-1.3) {$1$};
\node (2-) at (-2.3,0.3) {$2$};
\node (3-) at (-1.3,1.3) {$3$};
\node (4-) at (1.3,1.3) {$4$};
\node (5-) at (2.3,0.3) {$5$};
\node (6-) at (1.3,-1.3) {$6$};
\draw (1) to (2);
\draw (2) to (3);
\draw (3) to (4);
\draw (4) to (5);
\draw (5) to (6);
\draw (6) to (1);
\draw (1) circle(0.23);
\draw (2) circle(0.23);
\draw (3) circle(0.23);
\draw (4) circle(0.23);
\draw (5) circle(0.23);
\draw (6) circle(0.23);
\end{tikzpicture}
\caption{A plugging module used in Example \ref{exam:P5plugtoC6}.}
\label{fig:P5plugtoC6}
\end{figure}
\end{example}

\section{Eigenvectors for Special Graphs}
\label{sec:PnCnKn}

To a given $\mu$, applying the techniques introduced in Section \ref{sec:package}, one may obtain an eigenvector with as many nodal domains as possible for some special graphs, like path graphs $P_n$,
cycle graphs $C_n$ and complete graphs $K_n$.

Let $G=(V, E)$ be a graph, and let $\vec\phi=(x_1, \cdots, x_n)$ be a vector. A positive strong nodal domain of $\vec\phi$ is a maximal connected induced subgraph of $G$ on $D^+$.
For an eigenvalue $\mu$, it is possible to have many eigenvectors. The set of  eigenvectors with the same eigenvalue $\mu$ is denoted by $\mathcal{K}_\mu$. According to Theorem \ref{th:equivalent-binary}, there must be $\vec\phi\in \mathcal{K}_\mu$ such that $S(\vec\phi)=1$. We are led to introduce the following definition.

\begin{defn}
\label{def:nu}
We define
$$
\nu(\mu, G)=\max_{\phi\in \mathcal{K}_\mu} S(\phi),
$$
i.e., the largest number of strong nodal domains for all eigenvectors in $\mathcal{K}_\mu$.
\end{defn}

If $G$ is connected, then we have $\mathcal{K}_0=\{\hat{\vec 1}_G\}$ and thus $\nu(0, G)=1$. So we are just interested in counting $\nu(\mu, G)$ for $\mu\neq 0$. Before that, we
let $\sigma(G)$ denote the spectrum of $\Delta_1$ on a graph $G$.

\subsection{Path graphs $P_n$}

\begin{lemma}
\label{lem:Pn}
Let $H$ be a connected subgraph of $P_n$ and $V(P_n)$ be the vertex set of $P_n$.  Assume $V^o_H$ is not adjacent to $V(P_n)\backslash V_H$.
If $ V^o_H\cap\{1, n\}=\emptyset$, then for any $l \in \mathbb{N}^+$,  $(H, \hat{\vec 1}_{V^o_H})$ is a $\frac{1}{l}$-eigencomponent if and only if $V^o_H=\{m, m+1, \cdots, m+l-1\}$ with $m>1$ and $m+l\le n$. Otherwise, for any $k\in \mathbb{N}$, $(H, \hat{\vec 1}_{V^o_H})$ is a $\frac{1}{2k+1}$-eigencomponent if and only if $V^o_H=\{1, 2, \cdots, k+1\}$ or $V^o_H=\{n-k, n-k+1,  \cdots, n\}$.
\end{lemma}

\begin{proof}
First we assume $V^o_H=\{m, m+1, \cdots, m+l\}$ with $m>1, m+l<n$. To the vector $\vec\phi=\hat{\vec 1}_{V^o_H}$, set $z_{m+i, m+i-1}=-1+\frac{2i}{l}, i=1,2, \cdots, l-1.$ Since $z_{m, m+1}=z_{m+l+1, m+l}=-1$, it is easy to verify that $\vec\phi$ satisfies the eigensystem with $\mu=\frac{1}{l}$ on $V^o_H$. Therefore $(H, \hat{\vec 1}_{V^o_H})$ is a $\frac{1}{l}$-eigencomponent. Similarly, the verification applies to $V^o_H=\{1, 2, \cdots, k+1\}$ or $V^o_H=\{n-k, n-k+1,  \cdots, n\}$.

Conversely, let $(H, \vec\phi)$ be a $\frac{1}{l}$-eigencomponent. Since $H$ is a connected subgraph of $P_n$, $E_H$ consists of edges with neighboring  vertices. Since $V_H^o$ is connected, vertices in $V_H^o$ must be permuted consecutively, i.e., $V_H^o=\{ m, m+1, \cdots, p\}$, where $m\ge 1, p\in \{m,m+1,\cdots,n\}$. Reversing the order of the vertices one obtains the same eigencomponent, so we only need to distinguish two cases: $m=1, p<n$ and $m>1, p<n$. 
It is assumed that $(H, \vec\phi)$ is an eigencomponent, so $\vec\phi=\hat{\vec 1}_{V_H^o}$.

If $m=1, p<n$,
then we have the system
\begin{equation*}\left\{\begin{array} {l}
z_{12}=\frac{1}{l} ,\\
z_{21}+z_{23}=\frac{2}{l},\\
\cdots\\
z_{p(p-1)}+z_{p(p+1)}=\frac{2}{l}.
\end{array}\right.
\end{equation*}
It follows
$$
1 = z_{p(p+1)}=\frac{2p-1}{l}.
$$
Therefore $l=2p-1$, let $k=p-1$, then $l=2k+1$ and $V^o_H=\{1, 2, \cdots, k+1\}$.

It remains to consider the case where $1<m\leq p<n$.
The eigensystem becomes
\begin{equation*}\left\{\begin{array} {l}
z_{m(m-1)}+z_{m(m+1)}= \frac{2}{l},\\
z_{(m+1)m}+z_{(m+1)(m+2)}=\frac{2}{l},\\
\cdots\\
z_{p(p-1)}+z_{p(p+1)}=\frac{2}{l}.
\end{array}\right.
\end{equation*}
It follows
$$ 2=\frac{2(p-m+1)}{l}.$$
Therefore, $p=m+l-1$, i.e., $V^o_H=\{m, m+1, \cdots, m+l-1\}$.

The results for $p=n$ can be referred to those for $m=1$ and   lead to $V^o_H=\{n-k, n-k+1, \cdots, n\}$ by reversing the order of the vertices. Thus the proof is completed.
\end{proof}

In Lemma \ref{lem:Pn}, there are obviously two types of eigencomponents, both denoted by $(H,\hat{\vec1}_{V^o_H})$:

\begin{equation}
\label{eq:I-II}
\begin{array}{ll}
(A).&V^o_H=\{1, 2, \cdots, k+1\},\,\, V^s_H=\{k+2\}; \text{ or } \\
&V^o_H=\{n-k, n-k+1,  \cdots, n\},\,\, V^s_H=\{n-k-1\};\\
(B).&V^o_H=\{m, m+1, \cdots, m+l-1\},\,\,V^s_H=\{m-1,m+l\},\,\,m>1, m+l\le n.
\end{array}
\end{equation}

Naturally, we ask:

{\sl Can these eigencomponents be extended to eigenvectors by extension on the whole graph $P_n$?}

The answer depends on how big $n$ is. More precisely, it depends on how big $m$ and $n-p$ are, where $p=m+l-1$. In fact, the following system holds
\begin{equation*}\left\{\begin{array} {l}
z_{12}\in \frac{1}{l}[-1, 1] ,\\
z_{21}+z_{23}\in \frac{2}{l}[-1, 1],\\
\cdots,\\
z_{(m-1)m}+z_{(m-1)(m-2)}=\frac{2}{l}[-1, 1],
\end{array}\right.
\end{equation*}
if $m\ge \frac{l+3}{2}$, i.e., $m\ge r+2$,
where $l=2r$ or $2r+1$. Therefore, for any $ r\in \mathbb{N}^+$,
$$
\vec\phi=\frac{1}{2r}(0,\cdots, 0,1,\cdots,1,0\cdots, 0),
$$
having $2r$ consecutive $1$ in the middle and two groups of $r+1$ consecutive $0$ on both sides,
is an eigenvector of $\Delta_1$ on $P_{4r+2}$ with $\mu=\frac{1}{2r}$. Applying Proposition \ref{pro:extension} to the module $(P_{4r+2},\vec\phi)$,
we are able to extend $(P_{4r+2},\vec\phi)$ to any $P_n$ with $n\ge 4r+2$ with the same eigenvalue $\mu=\frac{1}{2r}$. This also implies that the spectrum of $\Delta_1(P_n)$ given in \cite{Chang2015} is incomplete. It should be corrected as follows.

\begin{theorem}
The spectrum of $P_n$ is
$$\sigma(P_n)=\{0\}\cup \left\{\left.\frac{1}{2k+1}\,\right|\,\ 0\le k\le \left\lfloor\frac{n}{2}\right\rfloor-1\right\}\cup \left\{\left.\frac{1}{2r}\right|\, 1\le r\le \left\lfloor\frac{n-2}{4}\right\rfloor\right\}.$$
\end{theorem}

By joining, pasting and plugging the two types of eigencomponents, $(A)$ and $(B)$, given in Eq.~\eqref{eq:I-II}, we obtain various modules in $P_n$.

\begin{theorem}
\label{th:nodal-Pn}
We have
\begin{enumerate}[(1)]
\item
\begin{equation}\label{eq:Pn-nodal-1}
\nu(\frac{1}{2r+1}, P_n)=\left\lfloor\frac{n+2r}{2r+1}\right\rfloor,\,\,\, n\ge 2r+2,\,\,\,r=0,1,2,\cdots, \left\lfloor\frac{n}{2}\right\rfloor-1,
\end{equation}
and $\forall\, r, \,\exists\, \vec\phi\in \mathcal{K}_{\frac{1}{2r+1}}$ such that $S(\vec\phi)= \nu(\frac{1}{2r+1}, P_n)$. Such $\vec\phi$ can be constructed by type (A) eigencomponents on both ends and type (B) eigencomponents in between.
\item \begin{equation}\label{eq:Pn-nodal-2}
\nu(\frac{1}{2r}, P_n)=\left\lfloor\frac{n-2}{2r}\right\rfloor-1,\,\,\, n\ge 4r+2,\,\,\,r=1,2,\cdots, \left\lfloor\frac{n-2}{4}\right\rfloor,
\end{equation}
and $\forall\, r, \,\exists\, \vec\phi\in \mathcal{K}_{\frac{1}{2r}}$ such that $S(\vec\phi)=\nu(\frac{1}{2r}, P_n)$. Such $\vec\phi$ can be constructed by $\{0,0,\cdots,0\}$ on both ends and type (B) eigencomponents in between.
\end{enumerate}
\end{theorem}

\begin{proof}
According to Theorem \ref{th:equivalent-binary}, each eigenvector $\vec\phi\in \mathcal{K}_\frac{1}{2r+1}$ is homotopic to a $\frac{1}{2r+1}$-module. All these modules are subgraphs of $P_n$. If such a module contains an end vertex $\{1\}$ or $\{n\}$, then it is of type (A), otherwise of type (B). Let $k$ be the number of nodal domains of the eigenvector $\vec\phi$. Applying Proposition \ref{pro:mu-eigencomponent} and Lemma \ref{lem:Pn}, a nodal domain containing an end vertex, contains $r+1$ vertices, then it must be a type (A) module; otherwise, it contains $2r+1$ vertices and then it must be a type (B) module.
In consequence, we have
$$
2(r+1)+(2r+1)(k-2)\le n\,\,\,
\Rightarrow\,\,\,
 k\le \frac{n+2r}{2r+1}.$$
Therefore
\begin{equation}\label{eq:nu1}
\nu(\frac{1}{2r+1}, P_n)\le \left\lfloor\frac{n+2r}{2r+1}\right\rfloor.
\end{equation}

On the other hand, it is sufficient to construct an eigenvector $\vec\phi$ in $\mathcal{K}_\frac{1}{2r+1}$ such that $S(\vec\phi) = \left\lfloor\frac{n+2r}{2r+1}\right\rfloor$.
The construction is as follows:
\begin{equation}\label{eq:block1}
(A) \plug (B) \plug  (B) \plug  (B) \plug \cdots\plug (B)\paste\{0, \cdots, 0\}\paste (A),
\end{equation}
where the notations (A) and (B) represent the types of eigencomponents given in Eq.~\eqref{eq:I-II},
the number of $0$ in $\{0, \cdots, 0\}$ equals to
$n+2r-(2r+1)\left\lfloor\frac{n+2r}{2r+1}\right\rfloor$.
It is important to note that the signs of the eigenvector $\vec\phi$ are different on both sides of the plugging modules. And it also has different sign
before and after $\{0, \cdots, 0\}$.
In short, the signs of the eigenvector $\vec\phi$ change alternatively from block to block except for $\{0, \cdots, 0\}$.
It can be easily seen that the resulting vector in Eq.~\eqref{eq:block1} is an eigenvector of $\Delta_1$ on $P_n$ and the corresponding eigenvalue is $1/(2r+1)$. Connecting with Eq.~\eqref{eq:nu1}, the first part of Theorem \ref{th:nodal-Pn} is proved.


Let $k$ be the number of nodal domains of the eigenvector $\vec\phi$ corresponding to the eigenvalue ${1}/{2r}$. A similar discussion to the above leads to
\[
2(r+1)+k2r\le n
\,\,\,
\Rightarrow\,\,\,
k\le \left\lfloor\frac{n-2}{2r}\right\rfloor-1,
\]
and the construction
$$
\{0, \cdots, 0\}\paste (B) \plug  (B) \plug  (B) \plug \cdots\plug (B)\paste\{0, \cdots, 0\}.
$$
Therefore the second part of Theorem \ref{th:nodal-Pn} is proved.
\end{proof}

In order to make an intuitive picture of Theorem \ref{th:nodal-Pn} for the readers,
we list some cases for typical eigenvalues below and  cartoon the structure of corresponding eigenvectors in Fig.~\ref{fig:Pn-most-nodal}.
\begin{align}
\nu(1, P_n)&=n, \,\,\,\,\,\,\,\,\,\,\,\,\,\,\,\,\,\,\,\,\,\,\,\,n\ge 2;\,\,\,\,\nu(\frac12 , P_n)=\left\lfloor\frac{n-2}{2}\right\rfloor-1 \,\,\,\,n\ge 6;\nonumber\\
\nu(\frac13 , P_n)&=\left\lfloor\frac{n+2}{3}\right\rfloor, \,\,\,\,n\ge 4;\,\,\,\,\nu(\frac14 , P_n)=\left\lfloor\frac{n-2}{4}\right\rfloor-1 \,\,\,\,n\ge 10;\\
\nu(\frac15 , P_n)&=\left\lfloor\frac{n+4}{5}\right\rfloor, \,\,\,\,n\ge 6;\,\,\,\,\nu(\frac16 , P_n)=\left\lfloor\frac{n-2}{6}\right\rfloor-1 \,\,\,\,n\ge 14.\nonumber
\end{align}


\begin{figure}
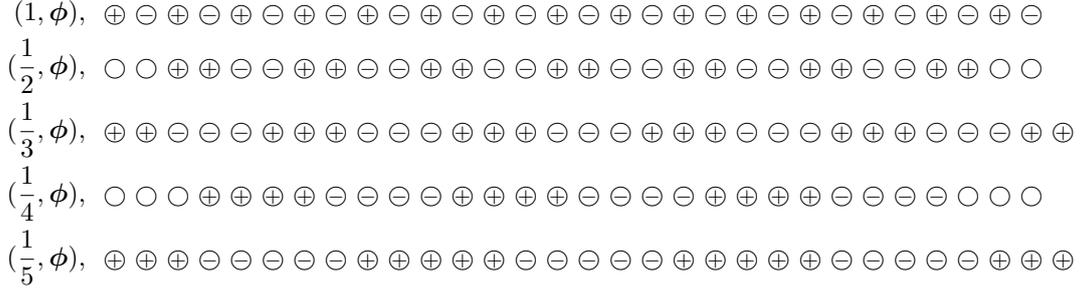

\begin{align*}
(1,\vec\phi),\,\,\,&\cirplus\cirminus\cirplus\cirminus\cirplus\cirminus\cirplus\cirminus\cirplus\cirminus\cirplus\cirminus\cirplus\cirminus\cirplus\cirminus\cirplus\cirminus\cirplus\cirminus\cirplus\cirminus
\cirplus\cirminus\cirplus\cirminus\cirplus\cirminus\cirplus\cirminus \\
(\frac12,\vec\phi),\,\,\,&\cir\cir\cirplus\cirplus\cirminus\cirminus\cirplus\cirplus\cirminus\cirminus\cirplus\cirplus\cirminus\cirminus\cirplus\cirplus\cirminus\cirminus\cirplus\cirplus\cirminus
\cirminus\cirplus\cirplus\cirminus\cirminus\cirplus\cirplus\cir\cir \\
(\frac13,\vec\phi),\,\,\,&\cirplus\cirplus\cirminus\cirminus\cirminus\cirplus\cirplus\cirplus\cirminus\cirminus\cirminus\cirplus\cirplus\cirplus\cirminus\cirminus\cirminus\cirplus\cirplus\cirplus\cirminus
\cirminus\cirminus\cirplus\cirplus\cirplus\cirminus\cirminus\cirminus\cirplus\cirplus \\
(\frac14,\vec\phi),\,\,\,&\cir\cir\cir\cirplus\cirplus\cirplus\cirplus\cirminus\cirminus\cirminus\cirminus\cirplus\cirplus\cirplus\cirplus\cirminus\cirminus\cirminus\cirminus
\cirplus\cirplus\cirplus\cirplus\cirminus\cirminus\cirminus\cirminus\cir\cir\cir \\
(\frac15,\vec\phi),\,\,\,&\cirplus\cirplus\cirplus\cirminus\cirminus\cirminus\cirminus\cirminus\cirplus\cirplus\cirplus\cirplus\cirplus\cirminus\cirminus\cirminus\cirminus\cirminus\cirplus\cirplus\cirplus\cirplus\cirplus
\cirminus\cirminus\cirminus\cirminus\cirminus\cirplus\cirplus\cirplus
\end{align*}
\caption{\small Eigenvectors of $\Delta_1(P_n)$ corresponding to $\mu=1,\frac12,\frac13,\frac14,\frac15$. All these eigenvectors
possess the largest possible number of strong nodal domains.}
\label{fig:Pn-most-nodal}
\end{figure}

\subsection{Cycle graphs $C_n$}

Now we turn to study cycle graphs $C_n$. It has been known that \cite{Chang2015}
$$
\sigma(C_n)=\left\{0, \frac{1}{\left\lfloor\frac{n}{2}\right\rfloor}, \cdots, \frac12, 1\right\}.
$$
Similarly, we have
\begin{lemma}
\label{lemma:Cn}
Let $H=(V_H,E_H)\subset C_n$ be a subgraph and $r\le \left\lfloor\frac n2\right\rfloor$.  Then $(H,\vec\phi)$ is a $\frac{1}{r}$-eigencomponent if and only if it is of the form:
\begin{align*}
V_H&=\{1,2, \cdots, r,\cdots, k\}, \,\, E_H=\{(12), (23), \cdots, ((k-1) k)\},\,\,k>r+1, \\
V_H^o&=\{m+1, m+2, \cdots, m+r\}, \,\,V_H^s=\{1,2, \cdots, m, m+r+1, \cdots, k\}, \,\,\phi=\hat{\vec 1}_{V_H^o},\,\,m>0.
\end{align*}
It should be noted that circulating the order of the vertices yields the same eigencomponent on $C_n$.
\end{lemma}

\begin{proof}
The proof is almost the same as that of Lemma \ref{lem:Pn}. In fact, the only difference lies on the fact that the type $(A)$ eigencomponent in Eq.~\eqref{eq:I-II} does not occur in cycle graphs.
\end{proof}

\begin{theorem}
\label{th:nodal-Cn}
We have
\begin{equation}\label{eq:nodal-Cn}
\nu(\frac{1}{r}, C_n)=2\left\lfloor\frac{n}{2r}\right\rfloor, \,\,\,n\ge 2r,\,\,\,r=1,2,\cdots, \left\lfloor\frac{n}{2}\right\rfloor,
\end{equation}
and $\forall\ r,\,\exists\,\vec\phi \in \mathcal{K}_\frac{1}{r}$ such that $S(\vec\phi)=\nu(\frac{1}{r}, C_n)$.
\end{theorem}


\begin{proof}
The proof is also similar to that of Theorem \ref{th:nodal-Pn}. The differences are as follows:
\begin{enumerate}
\item The spectrum of $\Delta_1(C_n)$ contains not only $\{\frac{1}{2r+1}\}$ but also $\{\frac{1}{2r} \}$.
\item Although the type $(A)$ eigencomponent does not occur in $C_n$, any eigencomponent on $C_n$ can be constructed by connecting the two type $(A)$ eigencomponents on $P_n$.
\end{enumerate}

To achieve the largest number of nodal domains,
we only consider the eigenvectors, constructed by joining, pasting, and plugging eigencomponents, have different signs on both sides.
Moreover, in view of the second difference above,
$\frac1r$-eigencomponents must appear in pairs,
and thus the number of nodal domains of an eigenvector must be even. Let an eigenvector $\vec\phi$ corresponding to $\mu=\frac1r$ have
$2k$ nodal domains. According to Lemma \ref{lemma:Cn}, we have
\[
2kr\le n \,\,\,
\Rightarrow \,\,\,
k\leq \left\lfloor\frac{n}{2r}\right\rfloor
\,\,\,
\Rightarrow
\,\,\,
2k\leq 2 \left \lfloor\frac{n}{2r}\right\rfloor,
\]
and then
$$
\nu(\frac{1}{r}, C_n)\le 2\left\lfloor\frac{n}{2r}\right\rfloor.
$$
On the other hand, it is sufficient to construct an eigenvector $\vec\phi$ in $\mathcal{K}_\frac{1}{r}$ such that $S(\vec\phi) = 2\left\lfloor\frac{n}{2r}\right\rfloor$. The construction is:
$$
(+B)\plug(-B)\plug(+B)\plug\cdots\plug(-B)\paste\{0, \cdots, 0\} \paste(+B),
$$
where the two end eigencomponents are the same,
the numbers of $\{0, \cdots, 0\}$ are in $[0, 2r-1]$ if $n$ is even, and in $[1, 2r]$ if $n$ is odd,
and $(\pm B)$ means the type $(B)$ eigencomponent with $\pm$ sign.
\end{proof}

Below we show some cases for typical eigenvalues in Theorem \ref{th:nodal-Cn}.
\begin{align*}
\nu(1, C_n)&=n, \;\;\;\;\;\;\,\,\,\,\,\,n\ge 2;\\
\nu(\frac12 , C_n)&=2\left\lfloor\frac{n}{4}\right\rfloor, \,\,\,\,n\ge 4;\\
\nu(\frac13 , C_n)&=2\left\lfloor\frac{n}{6}\right\rfloor, \,\,\,\,n\ge 6.
\end{align*}


\begin{remark}\rm
It is interesting to note that Theorems \ref{th:nodal-Pn} and \ref{th:nodal-Cn} for $P_n$ and $C_n$, respectively, are the counterparts of the Sturm-Liouville oscillation theorem in ODE, of the oscillatory eigenfunctions for $1$-Laplacian on intervals and circles \cite{Chang2009}, and of
the oscillatory eigenfunctions for standard Laplacian on $P_n$ in the linear spectral graph theory \cite{GantmacherKrein2002}.
\end{remark}

\subsection{Complete graphs $K_n$}
\label{sec:Kn}

Finally, we turn to study complete graphs $K_n$. The spectrum $\sigma(K_n)$ has been obtained in \cite{Chang2015}. But the proof there is too sketchy,
we rewrite here the proof in details in virtue of the new phenomena described in Theorem \ref{th:equivalent-binary}.

\begin{pro}
\label{pro:spectrum-Kn}
The spectrum of $K_n$ is
\begin{equation*}
\sigma(K_n)=\left\{\begin{array} {l}
\{ 0,\frac{n+1}{2(n-1)}, \cdots, \frac{n-2}{n-1}, 1\},\,\,\,\,\, n=3,5,7, \cdots,\\
\{ 0,\frac{n}{2(n-1)}, \cdots, \frac{n-2}{n-1}, 1\},\,\,\,\,\, n=2,4,6, \cdots.
\end{array}\right.
\end{equation*}
In particular, if $\hat{\vec 1}_D$ is an eigenvector with eigenvalue $\mu$, then $\mu$ is uniquely determined by the cardinal number of $D$.
\end{pro}

\begin{proof}
Firstly, we determine the eigenvalues.
According to the corollaries of Theorem \ref{th:equivalent-binary}, we may restrict ourselves to eigenvectors of the form $\hat{\vec 1}_D$, where $D$ is a connected subset and $\vol(D)\le \frac12 \vol(V)$. Since $d_1=\cdots =d_n=n-1$, one may reduce the volume $\vol(D)$ to the cardinal number $|D|$:
$\vol(D) = (n-1) |D|$. Accordingly, we only need to consider $\hat{\vec 1}_D$ with $|D|\le \frac n2$.
$\forall\, D,$ we shall prove:  $\vec\phi:=\hat{\vec 1}_D$ with $k:=|D|\le \frac n2$ is an eigenvector.
In fact, let $\mu_k=\frac{n-k}{n-1}$, $z_{ij}=0$ if $\phi_i=\phi_j$, we have
\begin{equation*}\begin{array} {l}
\sum_{j\sim i}z_{ij}(\vec\phi)=\sum_{j\neq i}z_{ij}(\hat{\vec 1}_D)=n-k=\mu_k (n-1)\sign(\phi_k),\,\,\,\forall\, i\in D,\\
\sum_{j\sim i}z_{ij}(\vec\phi)=-k \in \mu_k (n-1) [-1,1]=\mu_k (n-1)\sgn(\phi_k),\,\,\,\forall\, i\notin D,
\end{array}
\end{equation*}
provided $k\le n/2$. This proves that
$$
\sigma(K_n)=\{0\}\cup\{\frac{n-k}{n-1},\,\,,k=1,2,\cdots, \left\lfloor\frac n2\right\rfloor\}= \{0, \frac{1}{n-1}\left\lfloor\frac{n+1}{2}\right\rfloor, \cdots, \frac{n-2}{n-1}, 1\}.
$$
\end{proof}


\begin{theorem}
For any $n\in\{2,3,\cdots\}$ and any $\mu=\frac{n-k}{n-1}\in\sigma(K_n)$, we have
\begin{equation*}
\nu(\mu, K_n)=2.
\end{equation*}
\end{theorem}

\begin{proof}
Since any two vertices of $K_n$ are adjacent, this implies three requirements:
\begin{equation}\label{eq:requirements}
\begin{array}{ll}
\text{1.}&\text{Any subset $D$ is connected;}\\
\text{2.}&\text{All subsets with same number of vertices have the same structure;}\\
\text{3.}&\text{To any eigenvector, there is at most one positive and one negative nodal domains.}
\end{array}
\end{equation}
Accordingly, we have  $\nu(\mu, K_n)\le 2$. So, it is sufficient to construct an eigenvector $\vec\phi$ corresponding to $\mu$ such that $S(\vec\phi)=2$.
By Proposition \ref{pro:spectrum-Kn}, the two nodal domains possess the same cardinal number. Since $ 2k\le n$, we choose any two subsets $D^+$ and $D^-$ satisfying $|D^+|=|D^-|=k$ and $D^+\cap D^-=\emptyset$. Let
$$
\vec\phi=\hat{\vec 1}_{D^+} -\hat{\vec 1}_{D^-}.
$$
It can be easily verified that such $\vec\phi$
is indeed a $\mu$ eigenvector and $S(\vec\phi)=2$.
\end{proof}

\section{Courant nodal domain theorem}
\label{sec:nodal-domain}

As an extension of the Strum-Liouville theorem for ODE, the Courant nodal domain theorem is a fundamental result for elliptic partial differential equations on manifolds. Its discrete analog to the linear Laplacian on graphs is an important part of the linear spectral graph theory. In this section, we shall extend the nodal domain theorem to the nonlinear graph Laplacian.

Besides the strong nodal domain, we can still have a weak one. For an vector $\vec x=(x_1, \cdots, x_n)$, a positive weak nodal domain is a maximal connected induced subgraph of $G=(V,E)$ on a subset of $V$:
$$
\{i\in V\,|\, x_i\ge 0,\,\text{ and }\, \exists\, j\in V,\,\mbox{such that}\,x_j>0 \}.
$$
Similarly we can define the negative weak domain. Let $W(\vec x)$
denote the number of  weak nodal domains. Obviously, we have
$$
W(\vec x)\le S(\vec x).
$$

Let $G$ be a connected graph and
$$\lambda_1\le \lambda_2\le \cdots \le \lambda_k=\cdots=\lambda_{k+r-1}<\lambda_{k+r}\le \cdots \le \lambda_n$$
be the eigenvalues of the normalized Laplacian $L$ on $G$. Then the following extended Courant nodal domain theorem holds.

\begin{theorem}
[\cite{BiyikogluLeydoldStadler2007}, Theorem 3.1]
\label{th:2lap_nd}
Let $\vec\phi_k$ be an eigenvector corresponding to $\lambda_k$ with multiplicity $r$, then
$$
W(\vec\phi_k)\le k, \,\,\, S(\vec\phi_k)\le k+r-1.
$$
\end{theorem}

In particular, Fiedler showed that

\begin{cor}[\cite{Fiedler1975}, Lemma 3.2]
\label{cor:Fiedler}
 The eigenvector $\vec\phi$ affording to the smallest nonzero eigenvalue of any connected graph has $W(\vec\phi)=2$ weak nodal domains.
\end{cor}

Recently, Chang proved that those critical values defined in Eq.~\eqref{eq:def-ck} satisfy \cite{Chang2015}
\begin{equation}\label{eq:c1}
c_1\le c_2 \le \cdots \le c_n,
\end{equation}
and if
\begin{equation}\label{eq:c2}
c_{k+1}=\cdots =c_{k+l} = c,\,\,\, 0\le k\le k+l \le n,
\end{equation}
then the genus $\gamma(\mathcal{K}_c)\ge l$.

A critical value $c$ is said of {\sl topological multiplicity} $l$ if $\gamma(\mathcal{K}_c)=l$, denote by $\tm(c) = l$.

We therefore ask:

{\sl Do these results have their counterparts in the graph $1$-Laplacian?}

First, we extend Theorem \ref{th:2lap_nd} for strong nodal domains to $1$-Laplacian on graphs.

\begin{theorem}
\label{th:strong-nodal-domain}
 Let $(c_k, \vec\phi_k)$ be the $k$-th eigenpair of $\Delta_1$ and $r=\tm(c_k)$. Then we have
$$
S(\vec\phi_k)\le k+r-1.
$$
\end{theorem}

\begin{proof}  For simplicity, we omit the subscript of $\vec\phi_k$, and denote it by $\vec\phi$. Assume that $\vec\phi$ has a nodal domain decomposition:
\begin{equation*}
 \vec \phi=(\sum^{r^+}_{\alpha=1}\sum_{i\in D^+_\alpha}-\sum^{r^-}_{\beta=1}\sum_{i\in D^-_\beta}) \phi_i \vec e_i.
\end{equation*}
We consider the linear subspace
$$
E_{l}=\Span\{\vec1_{D_\alpha^+},\,\vec1_{D_\beta^-}: \alpha=1, \cdots, r^+,\,\beta=1, \cdots, r^-\},\,\,\,
l=S(\vec\phi)=r^++r^-.
$$
It is obvious that the genus $\gamma(E_{l}\cap X)=l$. Consequently, $\forall\, \vec\phi\in E_{l}\cap X$,
one has
\begin{equation}
\label{eq:1}
 \vec \phi=\sum^{r^+}_{\alpha=1} a_\alpha \vec1_{D_\alpha^+}-\sum^{r^-}_{\beta=1} b_\beta \vec1_{D_\beta^-}\,\, \text{ and } \,\,
\sum^{r^+}_{\alpha=1} |a_\alpha|\delta_\alpha^+ +\sum^{r^-}_{\beta=1} |b_\beta| \delta_\beta^- = 1,
\end{equation}
where $a_{\alpha}, b_{\beta}$ are the linear expansion coefficients in $E_{l}$.

We shall prove below: $I(\vec\phi)\le c_{k+r-1}$.


Let
\begin{align}
p^\pm_i&=\sum_{j\sim i, \, j\notin D^\pm_\gamma} z_{ij}(\vec\phi),\,\,\,\,\forall\, i\in D^\pm_\gamma,
\nonumber\\
Z^+_{\alpha\beta}&=\sum_{
\begin{subarray}{c}
j\sim i, \\
i\in D^+_\alpha,\,j\in D^-_\beta
\end{subarray}
}z_{ij}(\vec\phi) \geq 0,\,\,\,Z^\downarrow_\alpha=\sum_{
\begin{subarray}{c}
j\sim i, \\
i\in D^+_\alpha,\,j\in D^0
\end{subarray}
}z_{ij}(\vec\phi)\geq 0, \label{eq:Z+}\\
Z^-_{\beta\alpha}&=\sum_{
\begin{subarray}{c}
j\sim i, \\
i\in D^-_\beta,\,j\in D^+_\alpha
\end{subarray}
}z_{ij}(\vec\phi) \leq 0,\,\,\,Z^\uparrow_\beta=\sum_{
\begin{subarray}{c}
j\sim i, \\
i\in D^-_\beta,\,j\in D^0
\end{subarray}
}z_{ij}(\vec\phi) \leq 0.\label{eq:Z-}
\end{align}
Since
$$
\sum_{
\begin{subarray}{c}
j\sim i, \\
i\in D^\pm_\gamma,\,j\in D^\pm_\gamma
\end{subarray}
} z_{ij}(\vec\phi)=0, \,\,\, \forall\ \gamma,
$$
it follows
$$
\sum_{i\in D^\pm_\gamma}
p^\pm_i=\pm c_{k+r-1}\sum_{i\in D^\pm_\gamma}d_i=\pm c_{k+r-1}\delta^\pm_\gamma.$$
By summation over $D^+_\alpha$, we obtain
\begin{equation}
\label{eq:Z+sum}
\sum^{r^-}_{\beta=1} Z^+_{\alpha\beta}+Z^\downarrow_\alpha=
\sum_{i\in D_\alpha^+} p_i^+
=
c_{k+r-1}\delta^+_\alpha.
\end{equation}
Similarly,
\begin{equation}
\label{eq:Z-sum}
\sum^{r^+}_{\alpha=1} Z^-_{\beta\alpha}+Z^\uparrow_\beta=
\sum_{i\in D_\beta^-} p_i^-
=
-c_{k+r-1}\delta^-_\beta.
\end{equation}
Accordingly, we have
\begin{align*}
I(\vec\phi)
&=\sum_{\alpha,\beta}\sum_{
\begin{subarray}{c}
j\sim i, \\
i\in D^+_\alpha,\,j\in D^-_\beta
\end{subarray}
}|a_\alpha-b_\beta|
+\sum_\alpha
\sum_{
\begin{subarray}{c}
j\sim i, \\
i\in D^+_\alpha,\,j\in D^0
\end{subarray}
}|a_\alpha|
+\sum_\beta
\sum_{
\begin{subarray}{c}
j\sim i, \\
i\in D^-_\beta,\,j\in D^0
\end{subarray}
}|b_\beta| \\
&\le \sum_{\alpha,\beta}\sum_{
\begin{subarray}{c}
j\sim i, \\
i\in D^+_\alpha,\,j\in D^-_\beta
\end{subarray}
}(|a_\alpha|+|b_\beta|)
+\sum_\alpha
\sum_{
\begin{subarray}{c}
j\sim i, \\
i\in D^+_\alpha,\,j\in D^0
\end{subarray}
}|a_\alpha|
+\sum_\beta
\sum_{
\begin{subarray}{c}
j\sim i, \\
i\in D^-_\beta,\,j\in D^0
\end{subarray}
}|b_\beta|
\\&=\sum_{\alpha=1}^{r^+}|a_\alpha|(\sum^{r^-}_{\beta=1}Z^+_{\alpha\beta}+ Z^\downarrow_\alpha )-\sum_{\beta=1}^{r^-}|b_\beta|(\sum_{\alpha=1}^{r^+}Z^-_{\beta\alpha}+Z^\uparrow_\beta) \\
&= c_{k+r-1}(\sum_{\alpha=1}^{r^+}|a_\alpha| \delta_\alpha^+ + \sum_{\beta=1}^{r^-} |b_\beta| \delta_\beta^-) = c_{k+r-1},
\end{align*}
where Eqs.~\eqref{eq:Z+} and \eqref{eq:Z-} are used in the third line, and Eqs.~\eqref{eq:Z+sum}, \eqref{eq:Z-sum} and \eqref{eq:1} are applied in the last line.

Now, suppose the conclusion of the theorem is not true: $\textsl{S}(\vec\phi)>k+r-1$, i.e., $l\geq k+r$,  then
$$
c_{l}=\inf_{\gamma(A)\ge l}\sup\limits_{\vec x\in A}I(\vec x)\le \sup\limits_{\vec\phi\in E_{l}\cap X}I(\vec\phi)\le c_{k+r-1}\,\,\,
\Rightarrow \,\,\,
c_{k+r}\leq c_l \leq c_{k+r-1}.
$$
On the other hand, according to Eqs.~\eqref{eq:c1} and \eqref{eq:c2}, $r=\tm(c_k)$ implies
\[
c_k = c_{k+1} =\cdots = c_{k+r-1} < c_{k+r}.
\]
This is a contradiction.
\end{proof}


It remains to ask:

{\sl Does the estimates for weak nodal domains: $W(\vec\phi_k)\le k$ hold for $1$-Laplacian eigenvectors?}

A {\sl negative} answer is given below. Actually,
we present a graph on which $W(\vec\phi_2)=3$ holds in Example \ref{exam:10G-weak-nodal}.


\begin{example}\rm
\label{exam:10G-weak-nodal}
Let $G=(V,E)$ be a graph displayed in Fig.~\ref{fig:10G} with
\begin{align*}
V&=\{1,2,\cdots,10\}\\
E&=\{(1,2),(1,3),(1,4),(1,5),(2,3),(2,4),(3,4),(3,6),(5,7),(5,8), (7,8)
,(6,9),(6,10),(9,10)\}.
\end{align*}
Such graph is constructed by two copies of the $3$-rd order complete graph $K_3$ and a $4$-th order complete graph $K_4$ as shown in Fig.~\ref{fig:10G}.
Actually, we could consider two copies of the module $K_3\sim \circ$ and another module $\circ\sim K_4\sim \circ$ (also used in Example \ref{exam:6v}), all of which are $\frac17$-eigencomponents. By identifying the extra vertex adjacent to $K_3$ and one of the extra vertex adjacent to $K_4$, and then eliminate it, the first component is plugged in the second.  Similarly we do for the other side. That is, the graph is obtained by plugging as follows
$$
G = (K_3\sim \circ)\plug (\circ\sim K_4\sim \circ) \plug (\circ\sim K_3).
$$
One can readily verify that
$$\vec\phi=\frac{1}{28}(-1,-1,-1,-1,1,1,1,1,1,1)$$
is an eigenvector with respect to the eigenvalue $c_2=\frac17$ (because the Cheeger constant $h(G)=\frac17$), i.e., $\vec\phi\in\mathcal{K}_{c_2}$, but the number of weak nodal domains of $\vec\phi$ is $W(\vec\phi)=3$ (in fact $S(\vec\phi)=3$).
\begin{figure}[htpb]
\begin{center}
\begin{tikzpicture}[auto]
\node (1) at (-1,0) {$-$};
\node (2) at (0,1) {$-$};
\node (4) at (0,-1) {$-$};
\node (3) at (1,0) {$-$};
\node (5) at (-2,0) {$+$};
\node (6) at (2,0) {$+$};
\node (7) at (-4,1) {$+$};
\node (8) at (-4,-1) {$+$};
\node (9) at (4,1) {$+$};
\node (10) at (4,-1) {$+$};
\node (1-) at (-1.3,0.3) {$1$};
\node (2-) at (0.3,1.3) {$2$};
\node (4-) at (0.3,-1.3) {$4$};
\node (3-) at (1.3,0.3) {$3$};
\node (5-) at (-2.3,0.3) {$5$};
\node (6-) at (2.3,0.3) {$6$};
\node (7-) at (-4.3,1.3) {$7$};
\node (8-) at (-4.3,-1.3) {$8$};
\node (9-) at (4.3,1.3) {$9$};
\node (10-) at (4.3,-1.3) {$10$};
\draw (1) to (2);
\draw (1) to (3);
\draw (1) to (4);
\draw (1) to (5);
\draw (2) to (3);
\draw (2) to (4);
\draw (3) to (4);
\draw (3) to (6);
\draw (5) to (7);
\draw (5) to (8);
\draw (7) to (8);
\draw (6) to (9);
\draw (10) to (9);
\draw (6) to (10);
\draw (1) circle(0.23);
\draw (2) circle(0.23);
\draw (3) circle(0.23);
\draw (4) circle(0.23);
\draw (5) circle(0.23);
\draw (6) circle(0.23);
\draw (7) circle(0.23);
\draw (8) circle(0.23);
\draw (9) circle(0.23);
\draw (10) circle(0.23);
\end{tikzpicture}
\end{center}

\begin{center}
\begin{tikzpicture}[auto]
\node (1) at (-1,1) {$\bullet$};
\node (2) at (0,0) {$\bullet$};
\node (3) at (-1,-1) {$\bullet$};
\node (4) at (1, 0) {$0$};
\draw (1) to (2);
\draw (2) to (3);
\draw (1) to (3);
\draw (2) to (4);
\draw (1) circle(0.23);
\draw (2) circle(0.23);
\draw (3) circle(0.23);
\draw (4) circle(0.23);
\node (a) at (-1+5,0) {$\bullet$};
\node (b) at (0+5,1) {$\bullet$};
\node (c) at (0+5,-1) {$\bullet$};
\node (d) at (1+5,0) {$\bullet$};
\node (e) at (-2+5,0) {$0$};
\node (f) at (2+5,0) {$0$};
\draw (a) to (b);
\draw (a) to (c);
\draw (a) to (d);
\draw (a) to (e);
\draw (b) to (c);
\draw (b) to (d);
\draw (c) to (d);
\draw (d) to (f);
\draw (a) circle(0.23);
\draw (b) circle(0.23);
\draw (c) circle(0.23);
\draw (d) circle(0.23);
\draw (e) circle(0.23);
\draw (f) circle(0.23);
\end{tikzpicture}
\end{center}
\caption{\small A plugging module used in Example \ref{exam:10G-weak-nodal}.}
\label{fig:10G}
\end{figure}
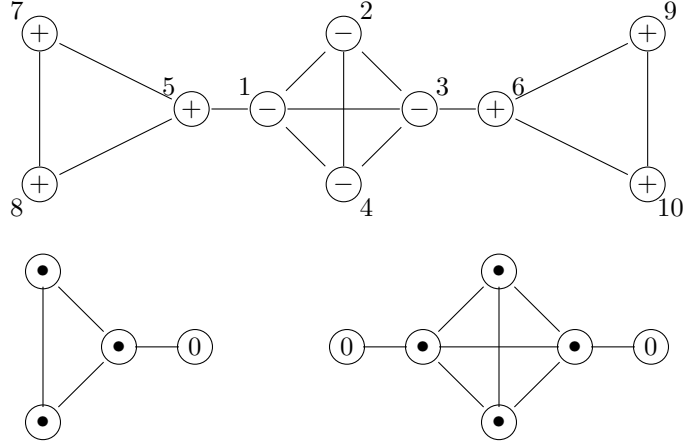

\end{example}

\section{Nodal Domains and the $k$-way Cheeger Constant}
\label{sec:k-way-Cheeger}

In this section we study the $k$-way Cheeger constant. Let $F_k$ denote the family of all $k$-disjoint subsets of $V$, $k=1,2, \cdots, n$, i.e.,
$$
{\sl F}_k=\{(S_1, \cdots, S_k)\,|\,S_i\subset V,\,S_i\cap S_j=\emptyset,\,\forall\,i\neq j\}.
$$
The $k$-way Cheeger constant defined in Eq.~\eqref{eq:hk} reads
$$
h_k=\min_{(S_1, \cdots, S_k)\in {\sl F}_k} \max_{1\le i\le k}\frac{|\partial S_i|}{\vol(S_i)}.
$$

Let
$$
\psi(S)=\frac{|\partial S|}{\vol(S)},\,\,\,
\forall\, S\subset V.
$$
We have
$$
\psi(S)=I(\hat{\vec 1}_S),
$$
and then
\begin{equation}
\label{eq:def-hk}
 h_k=\min_{(S_1, \cdots, S_k)\in {\sl F}_k} \max_{1\le i\le k}I(\hat{\vec 1}_{S_i}).
\end{equation}

\begin{lemma}
\label{lem:sumai=1}
For any $\vec a=(a_1, \cdots, a_k)\in \mathbb{R}^k$ and $(S_1, \cdots, S_k)\in {\sl F}_k$, if $\sum^k\limits_{i=1}a_i\hat{\vec 1}_{S_i}\in X$, then $\sum^k\limits_{i=1}|a_i|=1$.
\end{lemma}

\begin{proof}
Since $S_i\cap S_j=\emptyset$, as $i\neq j$,  $\forall\, x=(x_1, \cdots, x_n)\in X$, we have
$$1=\|x\|=\sum^n_{\alpha=1} d_\alpha|x_\alpha|=\sum^k_{i=1}\sum_{\alpha\in S_i}d_\alpha|x_\alpha|+\sum_{\alpha\notin \cup^k_{i=1}S_i}d_\alpha|x_\alpha|.$$
Then, it follows from $\|\hat{\vec 1}_{S_i}\|=1$ that
$$ 1=\left\|\sum^k_{i=1}a_i\hat{\vec 1}_{S_i}\right\|=\sum^k_{i=1}|a_i|\cdot\|\hat{\vec 1}_{S_i}\|=\sum^k_{i=1}|a_i|.$$
\end{proof}

\begin{theorem}
\label{th:ckhk}
Let $(c_k, \vec\phi_k)$ be the $k$-th eigenpair of $\Delta_1$. Then we have
\begin{enumerate}[(1)]
\item $c_k\le h_k,\,\,\,\forall\ k \in \{1,2,\cdots,n\}$;
\item If $S (\vec \phi_k)\ge m$, then $h_m\le c_k$.
\end{enumerate}
\end{theorem}

\begin{proof}
$\forall\ (S_1, \cdots, S_k)\in {\sl F}_k$, we define a linear $k$-dimensional space $E_k=\Span\{\hat{\vec 1}_{S_i}\,|\,1\le i\le k\}$. It is easy to see that the genus $\gamma(E_k\cap X)=k$. According to Lemma \ref{lem:sumai=1}, we have
\begin{align*}
c_k&=\inf_{\gamma(A)\ge k} \max_{\vec x\in A} I(\vec x)\\
&\le \min_{(S_1, \cdots, S_m)\in {\sl F}_k}\max_{\vec x\in E_k\cap X} I(\vec x)\\
&=\min_{(S_1, \cdots, S_m)\in {\sl F}_k}\max_{\sum^k_{i=1}|a_i|=1}I(\sum^k_{i=1}a_i\hat{\vec 1}_{S_i}).
\end{align*}
Noticing that $\vec x\to I(\vec x)$ is a semi-norm, we further have
\begin{align*}
c_k&\le \min_{(S_1, \cdots, S_k)\in {\sl F}_k}\max_{\sum^k_{i=1}|a_i|=1}\sum^k_{i=1}|a_i|I(\hat{\vec 1}_{S_i})\\
&\le \min_{(S_1, \cdots, S_k)\in {\sl F}_k}\max_{1\le i \le k}I(\hat{\vec 1}_{S_i})=h_k.
\end{align*}

To prove the second conclusion, we may assume that the $\vec \phi_k$ has nodal domains $\{D_\gamma \,|\, 1\le \gamma \le m\}$. Then, by Corollary \ref{cor:nodal-eigenvector},
we obtain
$$
\psi(D_\gamma)=I(\hat{\vec 1}_{D_\gamma})=c_k.
$$
Thus,
$$
h_m\le \max_{1\le \gamma \le m} I(\hat{\vec 1}_{D_\gamma})= c_k.
$$
The proof is completed.
\end{proof}

Chang \cite{Chang2015} proved that $c_2=h(G)$ and the corresponding eigenvector is equivalent to a Cheeger cut by an easy process. By this reason, the number of nodal domains of the eigenvector $\vec \phi_2$ corresponding to the first nonzero eigenvalue $c_2$ is of most concerned. In the linear spectral theory, it is well known that the eigenvector corresponding to the first nonzero eigenvalue $\lambda_2>0$ is of changing sign, i.e., $S(\vec \phi_2)\ge 2$, but this is not true for the $1$-Laplacian, as shown in Example \ref{exam:5G-one-nodal}.

\begin{example}\rm
\label{exam:5G-one-nodal}
We study the graph $G$ (see Example 2.5 in \cite{Chang2015}), in which $V=\{1,2,3,4,5\}$, and $E=\{(1,2), (1,3), (2,3), (3,4), (3,5)\}$. It is known that $c_2=\frac12 $ and $\vec \phi=\frac14 (1,1,0,0,0)$ is an eigenvector. We conclude that $\vec \phi$ is the only eigenvector with eigenvalue $\frac12 $, and thus we have $S(\vec \phi)=1$ and
$\nu(\frac12 , G)=1$.
\begin{figure}[htpb]\centering
\begin{tikzpicture}[auto]
\node (1) at (-1,1) {$+$};
\node (2) at (1,1) {$+$};
\node (3) at (0,0) {$0$};
\node (4) at (-1,-1) {$0$};
\node (5) at (1,-1) {$0$};
\node (1-) at (-1.3,1.3) {$1$};
\node (2-) at (1.3,1.3) {$2$};
\node (3-) at (0,0.4) {$3$};
\node (4-) at (-1.3,-1.3) {$4$};
\node (5-) at (1.3,-1.3) {$5$};
\draw (1) to (2);
\draw (1) to (3);
\draw (2) to (3);
\draw (3) to (4);
\draw (3) to (5);
\draw (1) circle(0.23);
\draw (2) circle(0.23);
\draw (3) circle(0.23);
\draw (4) circle(0.23);
\draw (5) circle(0.23);
\end{tikzpicture}
\caption{A graph used in Example \ref{exam:5G-one-nodal}.}
\label{fig:5G-one-nodal}
\end{figure}
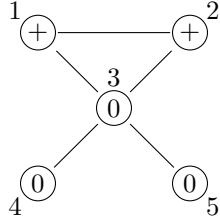

Now we give a detailed verification.
In fact, $z_{13}=z_{23}=1$, and after choosing $z_{12}=z_{34}=z_{35}=0$, we have
\begin{equation}
\label{eq:5G-one-nodal}
\left\{\begin{array} {l}
z_{12}+z_{13}=2\cdot\frac12  \sgn(x_1),\\
z_{21}+z_{23}=2\cdot\frac12  \sgn(x_2),\\
z_{31}+z_{32}+z_{34}+z_{35}\in 4\cdot \frac12 \cdot \sgn(x_3),\\
z_{43}\in \frac12 \cdot \sgn(x_4),\\
z_{53}\in \frac12 \cdot \sgn(x_5).
\end{array}\right.
\end{equation}
That is, $\vec\phi$ is an eigenvector with eigenvalue $\frac12$. Next, we shall show that $\vec\phi$ is the unique solution of the above system. In fact, let $\vec x=(x_1, \cdots, x_5)$ be a solution of the above system.

1. $x_4=0$. For otherwise, from the forth line in \eqref{eq:5G-one-nodal}, we obtain $x_3=x_4$. Again from the fifth line in \eqref{eq:5G-one-nodal} we have  $x_3=x_5$. And then
$\delta^\pm\ge 6>5=\frac{\vol(V)}{2}$, thus $\phi$ is not an eigenvector by Corollary \ref{cor:delta+<delta}. This is a contradiction.
Similarly we are able to show $x_5=0$.

2. $x_3=0$. If not, then either one of $\pm x_3>0$ holds. In this case, we show $x_1=0$. For otherwise, if $\pm x_1>0$, then again we have $\delta^\pm\ge 6$, it is impossible. In the remaining case, $\pm x_1<0$, it follows $z_{13}=\mp 1$, inserting into
the first equation in Eq.~\eqref{eq:5G-one-nodal}, we have $z_{12}=0$, i.e., $x_2=x_1$. Then $z_{23}=\mp 1$, inserting into the third line in \eqref{eq:5G-one-nodal}, we have
$$ z_{31}+z_{32}+z_{34}+z_{35}= \pm 4\notin 4\cdot \frac12\cdot [-1, 1].$$
Again, this is a contradiction. Similarly, we can show: $x_2=0$. However, $x_1=x_2=0$ can not be true. Because then we would have $z_{31}=z_{32}=z_{34}=z_{35}=\pm 1$, which contradicts with the third line in \eqref{eq:5G-one-nodal}.

In summary, $x_4=x_5=x_3=0$.

3. We show: $x_1=x_2\neq 0$.

First, we show: $x_1\neq 0$. For otherwise, $x_1=0$, in this case $\pm x_2>0$, it follows $z_{12}=\mp 1, \, z_{23}=\pm 1$, and then $z_{21}+z_{23}=\pm 2$, which contradicts with the second equation in \eqref{eq:5G-one-nodal}. Similarly, $x_2\neq 0$. Thus either $x_1x_2>0$ or $x_1 x_2<0$ holds. We shall prove: the case $x_1x_2<0$ is impossible. For otherwise, we may assume $x_1>0>x_2$, then $z_{12}=z_{13}=1$, which contradicts with  the first equation in \eqref{eq:5G-one-nodal}.

It remains the case $x_1x_2>0$, and then $x_1=x_2$. For otherwise, $z_{12}\neq 0$, but we have already known $z_{13}=1$. Again they contradicts with  the first equation in \eqref{eq:5G-one-nodal}. So we have completed the verification.
\end{example}

Naturally, we ask hereto:

{\sl Under what condition $S(\vec\phi_2)=1$ holds?}

From the definitions of the $k$-way Cheeger constant (see Eq.~\eqref{eq:hk}) and of the critical value (see Eq.~\eqref{eq:def-ck}), the following facts are obviously true:
\begin{equation}\label{eq:hk2}
h_1\le h_2 \le \cdots \le h_n=1,
\end{equation}
and
\begin{equation}\label{eq:h2lec2}
 h_2\le c_2,
\end{equation}
then we obtain
\begin{cor}\label{cor:c2=h2}
$$h_2=c_2.$$
\end{cor}

\begin{theorem}
\label{th:simple-c2}
If $c_2$ is topologically simple, i.e., $\gamma(\mathcal{K}_{c_2})=1$, then $\forall\, \vec\phi\in K_{c_2}$, $S(\vec\phi)\le 2$, i.e., $\nu(c_2, G)\le 2$.
\end{theorem}

\begin{proof}
Suppose not, i.e., there exists $\vec\phi\in \mathcal{K}_{c_2}$ such that $m:=S(\vec\phi)\ge 3$. By Eq.~\eqref{eq:hk2} and Theorem \ref{th:ckhk}, we have
$$ h_3\le h_m\le c_2\le c_3\le h_3,$$
Thus $c_2=c_3$. According to the Liusternik-Schnirelmann multiplicity theorem (\cite{Chang2015}, Theorem 4.13), we have
$\gamma(\mathcal{K}_{c_2})\ge 2$. This is a contradiction.
\end{proof}

\begin{remark}\rm Below we make several further remarks.
\begin{enumerate}
\item Example \ref{exam:10G-weak-nodal} shows that the topologically simple condition in Theorem \ref{th:simple-c2} cannot be dropped.
In fact, for the graph there, we have $c_2=\frac17$, $S(\vec\phi_2)=3$ and thus $\nu(c_2,G)\geq 3$.
Meanwhile, we can see that $\gamma(\mathcal{K}_{c_2})\geq 2$. This can be directly deduced from the nondecreasing property of the genus as well as the following facts
$$
S^1 \cong \{\left.a_1\hat{\vec 1}_{\{5,7,8\}}+a_2 \hat{\vec 1}_{\{6,9,10\}}\,\right||a_1|+|a_2|=1\}\subset \mathcal{K}_{c_2} \,\text{ and }\, \gamma(S^1)=2.
$$


\item We see that
$$ \phi= \frac{1}{28}(\hat{\vec 1}_{\{5,7,8\}}-\hat{\vec 1}_{\{1,2,3,4\}}+\hat{\vec 1}_{\{6,9,10\}})$$
is an eigenvector with $c_2=\frac17$ and $W(\vec\phi)=S(\vec\phi)=3$ in Example \ref{exam:10G-weak-nodal}, which also implies that there exists $
\vec \phi\in \mathcal{K}_{c_2}$ such that the number of weak nodal domains is larger than $2$.
That is, the result for the standard Laplacian: $W(\vec\phi)=2$ due to Fiedler (see  Corollary \ref{cor:Fiedler}) in the linear spectral graph theory cannot be extended to the graph $1$-Laplacian.

\item Theorem \ref{th:simple-c2} only gives a sufficient condition to guarantee $S(\vec\phi)\le 2$ instead of $S(\vec\phi)=1$. The question for the latter is still open.
\end{enumerate}
\end{remark}

\section{A Counterexample}
\label{sec:example}

For distinct eigenvalues $\{\mu_i\}\subset [0,1]$ of the graph $1$-Laplacian, it has been proved that the gap between them is at least $\frac{4}{n^2(n-1)^2}$ and thus the total number of different eigenvalues is possibly on the order of $\mathcal{O}(n^4)$ \cite{ChangShaoZhang2015}. A subset of eigenvalues $\{c_k\}$ has been chosen from them by using the minimax principle and can be ordered by the topological multiplicity \cite{Chang2015}. Moreover, the cardinal number of the subset is at least $n$ if counting topological multiplicity. So a natural question was raised by Chang \cite{Chang2015}:

{\sl Is there any eigenvalue $\mu$, which is not in the sequence: $\{c_1, c_2, \cdots, c_n\}$?}

An amazing graph of order $6$ displayed in Fig.~\ref{fig:counterexample} will give us a positive answer, though the numbers of different eigenvalues of $P_n$, $C_n$, $K_n$, are respectively $\frac34 n$, $\frac12 n$, $\frac12 n$ (see Section \ref{sec:PnCnKn}), and thus all less than $n$.
Actually, by Theorem \ref{th:equivalent-binary}, we shall show it has 9 different eigenvalues.


\begin{example}\rm
\label{exam:counterexampleG6}
Let $G_6=(V,E)$ be the graph shown in Fig.~\ref{fig:counterexample}.
We shall prove:
\begin{equation}\label{eq:G6}
\sigma(G_6) = \left\{0, \frac25, \frac59, \frac35, \frac23, \frac57, \frac34, \frac79,1\right\}.
\end{equation}
\begin{figure}[htpb]
\centering
\begin{tikzpicture}[auto]
\node (1) at (0,1) {$1$};
\node (2) at (0,-1) {$2$};
\node (3) at (-2,0) {$3$};
\node (4) at (-4,0) {$4$};
\node (5) at (2,-1) {$5$};
\node (6) at (2, 1) {$6$};
\draw (1) to (2);
\draw (1) to (3);
\draw (1) to (4);
\draw (1) to (5);
\draw (1) to (6);
\draw (2) to (3);
\draw (2) to (4);
\draw (2) to (5);
\draw (3) to (4);
\draw (5) to (6);
\draw (1) circle(0.23);
\draw (2) circle(0.23);
\draw (3) circle(0.23);
\draw (4) circle(0.23);
\draw (5) circle(0.23);
\draw (6) circle(0.23);
\end{tikzpicture}
\caption{A graph with $6$ vertices and $9$ different eigenvalues}
\label{fig:counterexample}
\end{figure}
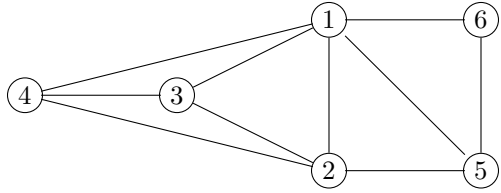
\end{example}

We need to search $(\mu,\vec\phi)$ by solving the system:
\begin{equation}\label{eq:n6sys}
\begin{cases}
z_{12}+z_{13}+z_{14}+z_{15}+z_{16}\in 5\mu \sgn(x_1),  \\
z_{21}+z_{23}+z_{24}+z_{25}\in 4\mu \sgn(x_2),  \\
z_{31}+z_{32}+z_{34}\in 3\mu \sgn(x_3),  \\
z_{41}+z_{42}+z_{43}\in 3\mu \sgn(x_4), \\
z_{51}+z_{52}+z_{56}\in 3\mu \sgn(x_5),  \\
z_{61}+z_{65}\in 2\mu \sgn(x_6).\\
\end{cases}
\end{equation}

Since for $\mu=1$, $\hat{\vec 1}_{\{i\}}$ with $i=1,2,\cdots, n$ is always an eigenvector, and for $\mu=0, \hat{\vec 1}_V$ is the unique eigenvector, we may assume $\mu\in (0,1)$.

According to Theorem \ref{th:equivalent-binary} and Corollary \ref{cor:nodal-eigenvector}, we may restrict ourselves to single nodal domain solutions, and then binary valued eigenvectors. More precisely, it is sufficient to check the nonzero binary vector $\hat{\vec 1}_A$ for some connected subset $A\subset V$. Again by Corollary \ref{cor:delta+<delta}, $\vol(A)\le \frac12 \vol(V)$ is a necessary condition. Therefore as a reduced subgraph, $A$ satisfies:
\begin{enumerate}[(C1)]
\item $\,\,\,\,\vol(A)\le \frac12 \vol(V)$.
\item $\,\,\,\,(A, E|_A)$ is connected.
\end{enumerate}

To the graph $G_6$, it is easily verified that $\vol(A)>10=\vol(V)/2$ for any subset $A$ with $|A|\ge 4$ (i.e. $A$ has at least four vertices).
It is reduced to study subsets $A$ with $|A|=2$ or  $3$.

First, we study the case $|A|=3$. We list all the possibilities of $A$ as follows, and then verify the system \eqref{eq:n6sys} case by case.

\begin{enumerate}
\item $A=\{1,2,3\},\{1,2,4\},\{1,2,5\},\{1,2,6\},\{1,3,4\},\{1,3,5\}, \{1,4,5\} $. For all these seven cases, we always have $\vol(A)>\frac12 \vol(V)$ and thus ignore them.\\
\item For $A=\{2,3,6\}, \{2,4,6\}, \{3,4,5\},\{3,4,6\},\{3,5,6\},\{4,5,6\}$, the reduced subgraphs are not connected, we ignore them.\\

\item $A=\{1,3,6\}$. We have $\mu=I(\hat{\vec 1}_A)=\frac35$,
and $z_{21}=z_{23}=z_{41}=z_{51}=z_{56}=z_{43}=-1$. Then after choosing
$z_{31}=z_{24}=z_{25}=z_{16}=-\frac15$,
\eqref{eq:n6sys} holds because
\begin{equation*}
\begin{cases}
1+\frac15+1+1-\frac15=3,  \\
-1-1-\frac15-\frac15\in [-\frac{12}{5},\frac{12}{5}],  \\
-\frac15+1+1=\frac95,  \\
-1+\frac15-1\in [-\frac95,\frac95], \\
-1+\frac15-1\in [-\frac95,\frac95],  \\
\frac15+1= \frac65. \\
\end{cases}
\end{equation*}
Since the positions of the vertices $3$ and $4$ are symmetric in the graph, the same is true for $A=\{1,4,6\}$.

\item $A=\{1,5,6\}$. We have $\mu=I(\hat{\vec 1}_A)=\frac25$,
and $z_{12}=z_{13}=z_{14}=z_{61}=z_{52}=1$. Then after choosing
$z_{15}=z_{34}=0$, $z_{23}=z_{24}=z_{56}=\frac15$, \eqref{eq:n6sys} holds because
\begin{equation*}
\begin{cases}
1+1+1+0-1= 2 ,  \\
-1+\frac15+\frac15-1\in [-\frac85,\frac85],  \\
-1-\frac15+0\in [-\frac65,\frac65],  \\
-1-\frac15+0\in [-\frac65,\frac65], \\
0+1+\frac15=\frac65,  \\
1-\frac15= \frac45. \\
\end{cases}
\end{equation*}

\item $A=\{2,3,4\}$. We have $\mu=I(\hat{\vec 1}_A)=\frac25$,
and $z_{12}=z_{13}=z_{14}=z_{61}=z_{52}=-1$. Then after choosing
$z_{15}=z_{34}=0$, $z_{23}=z_{24}=z_{56}=-\frac15$,
\eqref{eq:n6sys} holds because
\begin{equation*}
\begin{cases}
-1-1-1+0+1\in [-2,2],  \\
1-\frac15-\frac15+1= \frac85 ,  \\
1+\frac15+0= \frac65,  \\
1+\frac15+0= \frac65, \\
0-1-\frac15\in [-\frac65,\frac65],  \\
-1+\frac15\in [-\frac45,\frac45]. \\
\end{cases}
\end{equation*}

\item $A=\{2,3,5\}$. We have $\mu=I(\hat{\vec 1}_A)=\frac35$,
and $z_{21}=z_{24}=z_{31}=z_{51}=z_{56}=z_{34}=1$. Then after choosing
$z_{41}=z_{23}=z_{25}=z_{16}=\frac15$,
\eqref{eq:n6sys} holds because
\begin{equation*}
\begin{cases}
-1-1-\frac15-1+\frac15\in[-3,3],  \\
1+\frac15+1+\frac15= \frac{12}{5},  \\
1-\frac15+1= \frac95, \\
\frac15-1-1\in[-\frac95,\frac95],  \\
1-\frac15+1= \frac95,  \\
-\frac15-1\in [-\frac65,\frac65], \\
\end{cases}
\end{equation*}
By the same reason as in item 3, the same is true for $A=\{2,4,5\}$.

\item $A=\{2,5,6\}$. We have $\mu=I(\hat{\vec 1}_A)=\frac59$,
and $z_{21}=z_{51}=z_{61}=z_{23}=z_{24}=1$.  Then after choosing
$z_{34}=0$, $z_{52}=\frac79$ and $z_{13}=z_{14}=z_{65}=\frac19$,
\eqref{eq:n6sys} holds because
\begin{equation*}
\begin{cases}
-1+\frac19+\frac19-1-1\in [-\frac{25}{9},\frac{25}{9}],  \\
1+1+1-\frac79=\frac{20}{9},  \\
-\frac19-1+0\in [-\frac53,\frac53],  \\
-\frac19-1+0\in [-\frac53,\frac53], \\
1+\frac79-\frac19=\frac53,  \\
1+\frac19=  \frac{10}{9}, \\
\end{cases}
\end{equation*}

\end{enumerate}

In short, for all possible connected subsets $A$  with three vertices, there are three distinct $1$-Laplacian eigenvalues: $\mu=\frac35, \frac25, \frac59$ for $G_6$.

Next, we turn to investigate all possible subsets $A$  with two vertices. Before that, two lemmas below are needed.

\begin{lemma}
\label{Lem:minor} Let $G=(V,E)$ be a connected graph and $A\subset
V$. Assume that there exists $i_0\not\in A$ such that $j\in A$
whenever $j\sim i_0$. Then $\hat{\vec 1}_A$ is not an eigenvector
with eigenvalue $\mu<1$.
\end{lemma}

\begin{proof}
Suppose the contrary, i.e., $\hat{\vec 1}_A$ is an eigenvector.
Note that if $j\sim i_0$, then $j\in A$ and $z_{i_0j}(\hat{\vec 1})=-1$, which implies
\[-d_{i_0}=\sum_{j\sim i_0}z_{i_0j}\in\mu d_{i_0}\sgn( (\hat{\vec 1}_A)_{i_0})=[-\mu d_{i_0},\mu d_{i_0}].\]
But $\mu<1$, this is a contradiction.
\end{proof}

\begin{lemma}
\label{Lem:important}
Let $G=(V,E)$ be a connected graph and $A\subset V$ be a subset consisting of two adjacent vertices. If
\begin{equation}\label{eq:cond1}
\frac{|E(\{i\},A)|}{d_i}\le \frac{|\partial A|}{\vol(A)},
\quad \forall\ i\in A^c,
\end{equation}
then $\hat{\vec 1}_A$ is an eigenvector.
\end{lemma}

\begin{proof}
Assume
$A=\{u_1, u_2\}$. Then
$$
\mu = I(\hat{\vec 1}_A)=\frac{d_{u_1}+d_{u_2}-2}{d_{u_1}+d_{u_2}}=\frac{|\partial A|}{\vol(A)},\;\;
z_{u_1,i}(\hat{\vec 1}_A)=z_{u_2,i}(\hat{\vec 1}_A)=1,
\;\; i\notin A.
$$
The assumption is equivalent to
\[
-\mu d_i\le -|E(\{i\},A)|\le \mu d_i,\,\,\,\forall\, i\notin A.
\]

After choosing $z_{u_1, u_2}(\hat{\vec 1}_A)=\frac{d_{u_2}-d_{u_1}}{d_{u_1}+d_{u_2}}$ and $z_{ij}=0$ for $i\sim j$ with $i,j\in A^c$,
it is not difficult to verify that
\begin{equation*}
\begin{cases}
\displaystyle
\sum_{j\sim u_1}z_{u_1,j}=\frac{d_{u_2}-d_{u_1}}{d_{u_1}+d_{u_2}}+d_{u_1}-1=d_{u_1}\left(1-\frac{2}{d_{u_1}+d_{u_2}}\right)\in d_{u_1}\mu \sgn((\hat{\vec 1}_A)_{u_1}),  \\
\displaystyle
\sum_{j\sim u_2}z_{u_2,j}=\frac{d_{u_1}-d_{u_2}}{d_{u_1}+d_{u_2}} +d_{u_2}-1=d_{u_2}\left(1-\frac{2}{d_{u_1}+d_{u_2}}\right)\in d_{u_2}\mu \sgn((\hat{\vec 1}_A)_{u_2}),  \\
\displaystyle
\sum_{j\sim i}z_{ij}=-|E(\{i\},A)|\in [-\mu d_i,\mu d_i]=\mu d_i \sgn((\hat{\vec 1}_A)_i),  \quad i\in A^c,\\
\end{cases}
\end{equation*}
which implies that $\hat{\vec 1}_A$ is an eigenvalue
with the eigenvalue $\mu = I(\hat{\vec 1}_A)$.
\end{proof}

We list all the possibilities of $A$  as follows.
\begin{enumerate}
\item $A=\{2,6\}, \{3,5\},\{3,6\},\{4,5\},\{4,6\}$ are disconnected subsets, they should be ignored.
\item $A=\{1,2\}$.
Since
$$\frac{|E(i,A)|}{d_i}\le \frac23 < \frac79 =
\frac{|\partial A|}{\vol(A)}=I(\hat{\vec 1}_A)$$
holds for any $i\in A^c$.
Then we have $\mu = \frac79$ is an eigenvalue by Lemma \ref{Lem:important}.
\item $A=\{1,3\},\{1,4\}$.
For these two cases, it can be easily verified that
$$\frac{|E(i,A)|}{d_i}\le \frac23 < \frac34 =
\frac{|\partial A|}{\vol(A)}=I(\hat{\vec 1}_A)
$$ holds for any $i\in A^c$.
Then we have $\mu = \frac34$ is an eigenvalue by Lemma \ref{Lem:important}.
\item $A=\{1,5\}$, Since vertex $6\not\in A$ and the adjacent vertices of $6$ are $1$ and $5$, by Lemma \ref{Lem:minor}, $\hat{\vec 1}_A$ is not an eigenvector.
\item $A=\{1,6\},\{2,3\},\{2,4\},\{2,5\}$.
For all these four cases, it can be easily verified that
$$\frac{|E(i,A)|}{d_i}\le \frac23 < \frac57 =
\frac{|\partial A|}{\vol(A)}=I(\hat{\vec 1}_A)
$$ holds for any $i\in A^c$.
Then we have $\mu = \frac57$ is an eigenvalue by Lemma \ref{Lem:important}.
\item $A=\{3,4\}$.
Since
$$\frac{|E(i,A)|}{d_i}\le \frac23=\frac{|\partial A|}{\vol(A)}=I(\hat{\vec 1}_A)$$
holds for any $i\in A^c$,
then we have $\mu=\frac23$ is an eigenvalue by Lemma \ref{Lem:important}.
\item $A=\{5,6\}$. Since
$$\frac{|E(i,A)|}{d_i}\le \frac25 < \frac35=\frac{|\partial A|}{\vol(A)}=I(\hat{\vec 1}_A)$$
holds for any $i\in A^c$, we have
$\mu=\frac35$ is an eigenvalue by Lemma \ref{Lem:important}.
\end{enumerate}

In short, for all possible subsets $A$ mentioned in Lemma \ref{Lem:important} with two vertices, we have found five different $1$-Laplacian eigenvalues: $\mu=\frac79,\frac34, \frac57,\frac23, \frac35$.

Finally, together with $0$ and $1$, we have $9$ different
$1$-Laplacian eigenvalues in total as follow.
\begin{enumerate}
\item $(0, \hat{\vec 1}_V )$,
\item $(\frac25, \hat{\vec 1}_{\{1,5,6\}})$, $(\frac25, \hat{\vec 1}_{\{2,3,4\}})$,
\item $(\frac59, \hat{\vec 1}_{\{2,5,6\}})$,
\item $(\frac35, \hat{\vec 1}_{\{5,6\}})$, $(\frac35, \hat{\vec 1}_{\{2,3,5\}})$, $(\frac35, \hat{\vec 1}_{\{2,4,5\}})$,
$(\frac35, \hat{\vec 1}_{\{1,3,6\}})$, $(\frac35, \hat{\vec 1}_{\{1,4,6\}})$
\item $(\frac23, \hat{\vec 1}_{\{3,4\}})$,
\item $(\frac57, \hat{\vec 1}_{\{1,6\}})$, $(\frac57, \hat{\vec 1}_{\{2,3\}})$, $(\frac57, \hat{\vec 1}_{\{2,4\}})$, $(\frac57, \hat{\vec 1}_{\{2,5\}})$,
\item $(\frac34, \hat{\vec 1}_{\{1,3\}} )$$,(\frac34, \hat{\vec 1}_{\{1,4\}})$,
\item $(\frac79, \hat{\vec 1}_{\{1,2\}})$,
\item $(1, \hat{\vec 1}_{\{1\}}) $
\end{enumerate}

\begin{remark}\rm
We should point out that the graph displayed in Fig.~\ref{fig:counterexample} is the minimal possible graph such that the number of different eigenvalues of the graph is larger than the order of the graph. More precisely, after studied the graph with the order less than or equal to $5$ in the sense of graph isomorphism, we find that the number of different eigenvalues of such small graphs is less than or equal to the order of the corresponding graph. This proof is shown in Appendix.
\end{remark}

The following result is a direct consequence of the intersection theorem:

\begin{pro}
\label{pro:leqck}
For $k=1,2,\cdots,n$, we have
$$c_k\ge \sup\limits_{\dim Y\ge n+1-k}\inf\limits_{\vec x\in Y\cap X} I(\vec x),$$
where $Y$ is a linear subspace of $\mathbb{R}^n$.
\end{pro}

\begin{proof}
Assume $A\subset X\backslash\{0\}$ is symmetric with $\gamma(A)\ge k$. According to the intersection theorem, for any linear subspace $Y\subset \mathbb{R}^n$, with dimension $n+1-k$, we have $Y\cap A\neq \emptyset$. Thus
\begin{align*}
\sup\limits_{\vec x\in A}I(\vec x)\ge\inf\limits_{\vec x\in Y\cap X}I(\vec x),
\end{align*}
which implies
\begin{align*}c_k&=\inf\limits_{\gamma(A)\ge k}\sup\limits_{\vec x\in A}I(\vec x)
\ge \sup\limits_{\dim(Y)\ge n-1+k}\inf\limits_{\vec x\in A\cap X}I(\vec x).
\end{align*}
\end{proof}

Turn to Example \ref{exam:counterexampleG6}, we can prove that the sequence $\{c_1,\cdots,c_6\}$ only takes $4$ different values,
the verification of which is listed below.


\begin{enumerate}
\item The definition of $c_1$ leads to
$$
0\leq c_1=\inf\limits_{\gamma(A)\ge 0}\sup\limits_{\vec x\in A}I(\vec x)\le \sup\limits_{\vec x\in \{-\hat{\vec1}_V,\hat{\vec1}_V\}}I(\vec x)=0\;\Rightarrow\; c_1=0.
$$
\item According to Corollary \ref{cor:c2=h2}, we have $c_2=h_2=h(G)=\frac25$.
\item Together with Theorem \ref{th:ckhk} and Proposition \ref{pro:leqck},
we have
\begin{align*}
1\ge c_6\ge c_5\ge c_4\ge\inf\limits_{\vec x\in X_3}I(\vec x)=\frac{11}{13}\thickapprox 0.8461>\frac79, \\
\frac57=h_3 \ge c_3\ge\inf\limits_{\vec x\in X_4}I(\vec x)=0.6062>\frac35,
\end{align*}
where $X_4$ is spanned by
\begin{align*}
\vec e_1&=(-0.9392,0.6206,0.8583,-0.8248,-0.9064,0.7177),\\
\vec
e_2 &=(-0.5438, 0.8715,-0.6804,0.3908,-0.2052,-0.2375),\\
\vec
e_3 &=( 0.0121,0.5862,0.2137,-0.5873,0.7455,-0.9229),\\
\vec e_4 &=(0.0690,0.8525,-0.3608,-0.1948,-0.8312,-0.2769),
\end{align*}
and $X_3$ is spanned by
\[\vec e_1=(0,1,0,-1,0,0),\;\vec e_2=(0,0,1,-1,0,0),\;\vec e_3=(0,0,0,0,0,1).
\]
Hence we arrive at $c_3\in\{\frac23,\frac57\}$ and $c_4=c_5=c_6=1$ because each $c_k\ (k=1,2,\cdots,6)$ belongs to $\sigma(G_6)$ given in Eq.~\eqref{eq:G6}.
\end{enumerate}


\section{Algebraic Multiplicity}
\label{sec:algebra}

So far, we have only defined the topological multiplicity $\tm(\mu)=\gamma(\mathcal{K}_\mu)$ of an eigenvalue $\mu$ for $\Delta_1$.
In the following, we are going to define a fundamental eigenvector system to an eigenvalue $\mu$, and the respective multiplicity.

\begin{defn}
A system of eigenvectors $\{\vec \phi_1, \cdots, \vec \phi_m\}$ with respect to the eigenvalue $\mu$, where $\vec \phi_j=\hat{\vec 1}_{D_j},\,\,j=1,2,\cdots, m$, is said to be fundamental, if for any eigenpair $(\mu, \vec\psi)$ there exists $(a_1, \cdots, a_m)\in \mathbb{R}^m$ such that
$$
\vec\psi =  \sum^m_{i=1}a_i \vec \phi_i, \,\,\,\sum^m_{i=1}|a_i|=1,
$$
and neither one can be expressed linearly by other eigenvectors in the system.
\end{defn}

Let $\Lambda$  be the collection of all fundamental eigenvector systems with respect to $\mu$.
It can be partially ordered by inclusion, i.e.,
$S_1\prec S_2$ for $S_1,S_2\in\Lambda$, if $S_1\subset S_2$. Thus to each well ordered subset of $\Lambda$, there is a maximal element, which is called a {\sl maximal fundamental eigenvector system}. Let $\Lambda_0$ be the set of all maximal fundamental eigenvector systems with respect to $\mu$.

\begin{defn}
The largest cardinal number of  maximal fundamental eigenvector systems with respect to the eigenvalue $\mu$ is called the algebraic multiplicity of $\mu$, denoted by $\am(\mu)$, i.e.,
$$
\am(\mu)=\max_{S\in \Lambda_0}\card(S),
$$
and corresponding maximal fundamental eigenvector system is denoted by $\mathcal{S}_\mu$, i.e.,
\[
\mathcal{S}_\mu = \mathop{\argmax}\limits_{S\in \Lambda_0}\card(S).
\]
\end{defn}


For any graph $G=(V, E)$,
it can be easily seen that by Theorems 5.1 and 5.4 in \cite{Chang2015}:
\begin{enumerate}
\item $\am(1)=n$, $\mathcal{S}_1 = \{\hat{\vec 1}_{\{1\}},\cdots, \hat{\vec 1}_{\{n\}}\}$ is the unique maximal fundamental eigenvector system;
\item $\am(0)=1$, $\mathcal{S}_0 = \hat{\vec1}_{V}$ is the unique maximal fundamental eigenvector system.
\end{enumerate}
Actually, for any graph $G$, we also have $\tm(0)=1$, but no general conclusion for $\tm(1)$. That is,
$\am(\mu)$ is usually different from $\tm(\mu)$ when $\mu \in (0,1]$. As an example, we have calculated the algebraic multiplicity of each $\mu \in \sigma(G_6)$ already shown in Example \ref{exam:counterexampleG6},
and the results are listed in Table \ref{tab:am}.

\begin{table}
\caption{\small Algebraic multiplicity and
corresponding maximal fundamental eigenvector system for the graph $G_6$ shown in Fig. \ref{fig:counterexample}.}
\begin{center}
\begin{tabular}{c|c|l}
\hline\hline
$\mu$ & $\am(\mu)$ & $\mathcal{S}_\mu$ \\
\hline
$0$     &  $1$     & $\{\hat{\vec 1}_{V}\}$  \\
\hline
$\frac25$ & $2$    & $\{\hat{\vec 1}_{\{1,5,6\}}, \hat{\vec 1}_{\{2,3,4\}}\}$ \\
\hline
$\frac59$ & $1$   & $\{\hat{\vec 1}_{\{2,5,6\}}\}$ \\
\hline
$\frac35$ & $4$    & $\{\hat{\vec 1}_{\{5,6\}}, \hat{\vec 1}_{\{2,4,5\}}, \hat{\vec 1}_{\{1,4,6\}}, \hat{\vec 1}_{\{2,3,5\}}\}$ \\
\hline
$\frac23$ & $1$ & $\{\hat{\vec 1}_{\{3,4\}}\}$ \\
\hline
$\frac57$ & $4$ & $\{\hat{\vec 1}_{\{1, 6\}}, \hat{\vec 1}_{\{2,3\}}, \hat{\vec 1}_{\{2,4\}}, \hat{\vec 1}_{\{2,5\}}\}$ \\
\hline
$\frac34$ & $2$ & $\{\hat{\vec 1}_{\{1,3\}}, \hat{\vec 1}_{\{1,4\}}\}$ \\
\hline
$\frac79$ & $1$ & $\{\hat{\vec 1}_{\{1,2\}}\}$ \\
\hline
$1$       & $6$ & $\{\hat{\vec 1}_{\{i\}}:i=1,2,3,4,5,6\}$ \\
\hline\hline
\end{tabular}
\end{center}
\label{tab:am}
\end{table}

The number $\am(\mu)$ is more interesting than $\tm(\mu)$ in counting multiple eigenvectors.
Indeed,

\begin{theorem}
\label{th:algebra-nodal}
If $(\mu, \vec \phi)$ is an eigenpair, and $\vec \phi$ has a nodal domain decomposition,
then $\am(\mu)\ge S(\vec \phi)$.
\end{theorem}

\begin{proof}
Assume that
$$
\vec \phi=(\sum^{r^+}_{\alpha=1}\sum_{i\in D^+_\alpha}-\sum^{r^-}_{\beta=1}\sum_{i\in D^-_\beta}) x_i \vec e_i
$$
is a nodal domain decomposition. We consider the system
$$
S=\{\hat{\vec 1}_{D^+_\alpha}, \hat{\vec 1}_{D^-_\beta}:\, 1\le \alpha\le r^+,\,1\le \beta\le r^- \},
$$
and
all vectors of the form
$$
\vec\phi(a_1, \cdots, a_{r^+}, b_1,\cdots b_{r^-})=\sum^{r^+}_{\alpha=1} a_\alpha \hat{\vec 1}_{\alpha}-\sum^{r^-}_{\beta=1} b_\beta \hat{\vec 1}_{\beta}\,\, \text{ with } \,\,
\sum^{r^+}_{\alpha=1} |a_\alpha| +\sum^{r^-}_{\beta=1} |b_\beta|  = 1
$$
are eigenvectors with eigenvalue $\mu, \,\forall\, (a_1, \cdots, a_{r^+}, b_1,\cdots b_{r^-})\in \mathbb{R}^{r^++r^-}_+$.
Since the supports of those eigenvectors in the system are disjoint, neither one of them can be expressed linearly by others. That is,
$S$ is a fundamental eigenvector system.

Starting from $S$,
one can easily find a maximal fundamental eigenvector system $T$ with respect to the eigenvalue $\mu$, which contains $S$ as a subsystem. Therefore, we conclude:
$$
\am(\mu)\ge r^++r^-=S (\vec \phi),
$$
and then complete the proof.
\end{proof}

We shall provide another characterization of the algebraic multiplicity $\am(\mu)$.

\begin{theorem}
\label{th:algebra-span-eigenvector}
$$
\am(\mu)=\dim \Span(\mathcal{K}_\mu).
$$
\end{theorem}

\begin{proof}
Denote by $\widehat{\mathcal{K}}_\mu$ the set of all binary vectors in $\mathcal{K}_\mu$. On one hand, due to Theorem \ref{th:equivalent-binary}, $\widehat{\mathcal{K}}_\mu$ is a fundamental system; on the other hand, again by Theorem \ref{th:equivalent-binary}, every eigenvector in a fundamental eigenvector system $S$ is binary, we have $S\subset \widehat{\mathcal{K}}_\mu$.
Thus
\begin{equation}\label{eq:am}
\am(\mu)=\max\limits_{S\in\Lambda_0}\card(S)=\max\limits_{S\in\Lambda_0}\dim\Span(S)=\dim \Span(\widehat{\mathcal{K}}_\mu).
\end{equation}

Since $\widehat{\mathcal{K}}_\mu\subset \mathcal{K}_\mu$, it remains to prove: $\Span(\mathcal{K}_\mu)\subset \Span(\widehat{\mathcal{K}}_\mu)$.

For $\vec \phi=(x_1, \cdots, x_n)\in \mathcal{K}_\mu$, let $D^1,D^2,\cdots,D^k$ be the nodal domains of $\vec \phi$.
It is easily seen that
$$\vec \phi^i:=\frac{\vec \phi|_{D^i}}{\|\vec \phi|_{D^i}\|}\in \mathcal{K}_\mu,\,\,\, i=1,2,\cdots,k.$$
Let
$$\vec \phi^i=\sum_{j\in D^i}x^i_j\vec e_j\text { and }  D^i(t)=\{j\in D^i:|x^i_j|\ge t\},\;0\le t\le \max\limits_{j\in D^i}|x_j^i|.$$
Denote by $p_i$ the cardinality of $\{|x^i_j|:j\in D^i\}$, and assume  $\{|x^i_j|:j\in D^i\}=\{t_1,t_2,\cdots,t_{p_i}\}$, where $t_1<t_2<\cdots<t_{p_i}$. Then we have
\[D^{i,1}\supsetneqq D^{i,2} \supsetneqq\cdots \supsetneqq D^{i,p_i}\]
where $D^{i,1}:=D^i(t_1)=D^i$, $D^{i,2}:=D^i(t_2)$, $\cdots$, $D^{i,p_i}:=D^i(t_{p_i})$.

Let $\vec x^{i,s}=\vec{\hat{\vec 1}}_{D^{i,s}}\cdot \sign(D^i)$, $s=1,2,\cdots,p_i$, where \[\sign(D^i)=\begin{cases}
1,& \text{ if }D^i\text{ is the positive nodal domain}\\
-1,& \text{ if }D^i\text{ is the negative nodal domain}
\end{cases}\]

With the aid of the obvious facts $\sgn(x^{i,s}_u-x^{i,s}_v)\supset \sgn( \phi^{i}_u- \phi^{i}_v)$ and $\sgn(x^{i,s}_v)\supset \sgn(x^{i}_v)$ for any $u,v\in V$, it can be verified without effort that $\vec x^{i,s}\in \mathcal{K}_\mu$ and so $\vec x^{i,s}\cdot \sign(D^i)\in \widehat{\mathcal{K}}_\mu$.

However, we have
\begin{align*}
\vec \phi^i&=(t_1\vec 1_{D^{i,1}\setminus D^{i,2}}+t_2\vec 1_{D^{i,2}\setminus D^{i,3}}+\cdots+t_{p_i}\vec 1_{D^{i,p_i}})\cdot \sign(D^i)
\\&=(t_1\vol(D^{i,1}\setminus D^{i,2})\vec{\hat{\vec 1}}_{D^{i,1}\setminus D^{i,2}}+t_2\vol(D^{i,2}\setminus D^{i,3})\vec{\hat{\vec 1}}_{D^{i,2}\setminus D^{i,3}}+\cdots+t_{p_i}\vol(D^{i,p_i})\vec{\hat{\vec 1}}_{D^{i,p_i}})\cdot \sign(D^i)
\\&=t_1 \vol(D^{i,1}) \vec x^{i,1}+(t_2-t_1)\vol(D^{i,2}) \vec x^{i,2}+\cdots+(t_{p_i}-t_{p_i-1})\vol(D^{i,p_i}) \vec x^{i,p_i},
\end{align*}
and
\begin{align*}
&t_1 \vol(D^{i,1})+(t_2-t_1)\vol(D^{i,2})+\cdots+(t_{p_i}-t_{p_i-1})\vol(D^{i,p_i})
\\=&t_1\vol(D^{i,1}\setminus D^{i,2})+t_2\vol(D^{i,2}\setminus D^{i,3})+\cdots+t_{p_i}\vol(D^{i,p_i})=1,
\end{align*}
i.e., $\vec \phi^i\in \conv(\vec x^{i,1},\vec x^{i,2},\cdots,\vec x^{i,p_i})$. This implies that
\begin{equation}\label{eq:trianglex}
\vec \phi\in \conv(\vec \phi^1,\vec \phi^2,\cdots,\vec \phi^k)\subset \conv(\vec x^{i,s}:i=1,2,\cdots,k,~~s=1,\cdots,p_i):=\triangle(\vec \phi).
\end{equation}

Thus, $\vec \phi\in\triangle(\vec \phi)\subset \Span(\widehat{\mathcal{K}}_\mu)$, and then $\mathcal{K}_\mu\subset \Span(\widehat{\mathcal{K}}_\mu)$, therefore, $\Span(\mathcal{K}_\mu)\subset \Span(\widehat{\mathcal{K}}_\mu)$. The proof is completed.
\end{proof}

\begin{remark}\rm We make some relevant remarks below.
\begin{enumerate}
\item In the above proof, $\triangle(\vec \phi)\subset \mathcal{K}_\mu\subset \Span(\widehat{\mathcal{K}}_\mu)$ holds for
any $\vec\phi\in\mathcal{K}_\mu$.

\item Let $C_{\textit{l}}=\{\vec \phi(a_1, \cdots, a_{r^+}, b_1,\cdots b_{r^-}),\,|\,(a_1, \cdots, a_{r^+}, b_1,\cdots b_{r^-})\in \mathbb{R}^{r^++r^-}_+ \} $ be the positive cone of the  $l=r^++r^-$ dimensional linear subspace $E_{l}=\Span\{\hat{\vec 1}_{D^+_\alpha}, \hat{\vec 1}_{D^-_\beta}:\, 1\le \alpha\le r^+,\,1\le \beta\le r^- \}$.
The set $C_{l}\cap X$ is contractible, $\gamma(C_l\cap X)=1$, i.e., its contribution to $\tm(\mu)$ is $1$, but $\am(\mu)\ge S (\vec \phi)=r^++r^-$.
\end{enumerate}
\end{remark}

\section{Optimal Cheeger Cut}
\label{sec:optimal}

There are many papers dealing with the Cheeger cut in literatures. In particular, its algorithms are mostly concerned. Let $G=(V, E)$, the so called Cheeger cut is the subset of vertices $S\subset V$ which satisfies
$$
h(G)=\min_S \frac{|E(S, S^c)|}{\min\{\vol(S), \vol(S^c)\}}.
$$
However, to all cuts $S$, the vector $\hat{\vec 1}_S$ is an eigenvector corresponding to the Cheeger constant $\mu_2=c_2=h(G)$ and vice versa \cite{Chang2015}.
The size of $\mathcal{K}_{c_2}$ may be very large due to Theorem \ref{th:equivalent-binary}. In fact, as we pointed out in Theorem \ref{th:equivalent-binary} as well as in Corollary \ref{cor:nodal-eigenvector}, if $\vec \phi\in \mathcal{K}_{c_2}$ and if $\{D^\pm_\gamma\}$ are all nodal domains of $\vec \phi$, then $\hat{\vec 1}_{D^\pm_\gamma}$ are eigenvectors, and thus all $\{D^\pm_\gamma\}$ are  Cheeger cuts. We ask:

{\sl
Facing so many Cheeger cuts, which one should we take?
}

The (normalized) characteristic function of the subset $\hat{\vec 1}_S$ is one to one correspondent to the subset $S$, therefore we may restrict ourselves to binary eigenvectors. Let us denote \[\B =\{\vec \phi: V\to \{0,c\},\,\,\,\left| \|\vec \phi\|=1\right.\}.\]

Since the realistic purpose in finding the Cheeger cut is to divide a graph into two parts as equally as possible, we propose the following definition.

\begin{defn}
Let $\vec \phi_0$ be a binary eigenvector with respect to $h(G)$ satisfying
$$\delta_0(\vec \phi_0)=\min\limits_{\vec \phi\in \mathcal{K}_{c_2}\cap \B }\{\delta_0(\vec \phi)\},$$
where $\delta_0(\vec \phi)=\vol(D^0)$. We call $S= D^0(\vec \phi_0)$ an optimal Cheeger cut.
\end{defn}

For a given Cheeger cut, according to Theorem 2.11 in \cite{ChangShaoZhang2015}, the part with larger volume must be connected.

To find the optimal Cheeger cut, we should first compute all the
binary eigenvectors of the second eigenvalue and then choose the one
with the minimal volume of $D^0$. Precisely, the optimization
problem can be formularized as
\begin{equation}
\label{eq:zeronorm}
\begin{array}{c}
\max \limits_{\vec x}\|\vec x\|_0
\\
\text{subject to: }
\vec x = \mathop{\argmin}\limits_{\vec y\in\pi}I(\vec y),
\end{array}
\end{equation}
where $\|\vec x\|_0:=\delta^+(\vec x)+\delta^-(\vec x)$ is the weighted zero norm of $\vec x$, i.e. the total degree of nonzero components.
It is easily seen that seeking the optimal Cheeger cut employs a kind of dense representation method in contrast to the popular sparse representation method for image processing.


The first graph for the optimal Cheeger cut is that used in Example 2.1 in \cite{ChangShaoZhang2015}. The Cheeger value of the graph is $h_2=\frac{1}{5}$.
Let $A_1=\{1,2,3\}$, $A_2=\{11,12,13\}$, $A=A_1\sqcup A_2$. It can be easily verified that,
there are only two kinds of Cheeger cuts, $(A,A^c)$ (i.e., $\vec\phi=\hat{\vec 1}_A$) and $(A_1,A_1^c)$
(i.e., $\vec\phi=\hat{\vec 1}_{A_1}$), in the sense of graph isomorphism.
Therefore, $(A,A^c)$ is an optimal Cheeger cut, but $(A_1,A_1^c)$ is not, because
$\vol(A^c)<\vol(A_1^c)$.

The second example below provides a more clear picture of the optimal Cheeger cut.

\begin{example}\rm
\label{exam:7-G}
Let
$$
G=(V, E), \,\,\,
V=\{1,2,3,4,5,6,7\},\,\,\,E=\{(12), (23), (34), (45), (36), (37)\}
$$
be the graph. It is easy to verify that $h_2=\frac13$. The eigenvectors are constructed by two eigencomponents shown in Fig.~\ref{fig:7-G}.
\begin{figure}[htpb]
\begin{center}
\begin{tikzpicture}[auto]
\node (1) at (0,0) {$\bullet$};
\node (2) at (1,0) {$\bullet$};
\node (3) at (2,0) {$0$};
\draw (1) to (2);
\draw (2) to (3);
\draw (1) circle(0.23);
\draw (2) circle(0.23);
\draw (3) circle(0.23);
\end{tikzpicture}
\end{center}

\begin{center}
\begin{tikzpicture}[auto]
\node (2) at (2,0) {$0$};
\node (3) at (3,0) {$\bullet$};
\node (4) at (4,0) {$0$};
\node (6) at (3,1) {$\bullet$};
\node (7) at (3,-1) {$\bullet$};
\draw (2) to (3);
\draw (3) to (4);
\draw (3) to (6);
\draw (3) to (7);
\draw (2) circle(0.23);
\draw (3) circle(0.23);
\draw (4) circle(0.23);
\draw (6) circle(0.23);
\draw (7) circle(0.23);
\end{tikzpicture}\end{center}
\caption{A $\frac13$-eigencomponent with two sockets used in Example \ref{exam:7-G}.}
\label{fig:7-G}
\end{figure}

By extensions, we obtain binary eigenvectors: $\hat{\vec 1}_{\{1,2\}},
\hat{\vec 1}_{\{4,5\}}, \hat{\vec 1}_{\{3,6,7\}}, \hat{\vec 1}_{\{1,2,4,5\}}$, which are shown in Fig.~\ref{fig:7-G-2}. Actually, there are only these four possible Cheeger cuts.

\begin{figure}[htpb]
\begin{center}
\begin{tikzpicture}[auto]
\node (1) at (1,0) {$\bullet$};
\node (2) at (2,0) {$\bullet$};
\node (3) at (3,0) {$0$};
\node (4) at (4,0) {$0$};
\node (5) at (5,0) {$0$};
\node (6) at (3,1) {$0$};
\node (7) at (3,-1) {$0$};
\node (1-) at (1.3,0.3) {$1$};
\node (2-) at (2.3,0.3) {$2$};
\node (3-) at (3.3,0.3) {$3$};
\node (4-) at (4.3,0.3) {$4$};
\node (5-) at (5.3,0.3) {$5$};
\node (6-) at (3.3,1.3) {$6$};
\node (7-) at (3.3,-1.3) {$7$};
\draw (1) to (2);
\draw (2) to (3);
\draw (3) to (4);
\draw (4) to (5);
\draw (3) to (6);
\draw (3) to (7);
\draw (1) circle(0.23);
\draw (2) circle(0.23);
\draw (3) circle(0.23);
\draw (4) circle(0.23);
\draw (5) circle(0.23);
\draw (6) circle(0.23);
\draw (6) circle(0.23);
\draw (7) circle(0.23);
\end{tikzpicture}
\hspace{1cm}
\begin{tikzpicture}[auto]
\node (1) at (1,0) {$0$};
\node (2) at (2,0) {$0$};
\node (3) at (3,0) {$0$};
\node (4) at (4,0) {$\bullet$};
\node (5) at (5,0) {$\bullet$};
\node (6) at (3,1) {$0$};
\node (7) at (3,-1) {$0$};
\node (1-) at (1.3,0.3) {$1$};
\node (2-) at (2.3,0.3) {$2$};
\node (3-) at (3.3,0.3) {$3$};
\node (4-) at (4.3,0.3) {$4$};
\node (5-) at (5.3,0.3) {$5$};
\node (6-) at (3.3,1.3) {$6$};
\node (7-) at (3.3,-1.3) {$7$};
\draw (1) to (2);
\draw (2) to (3);
\draw (3) to (4);
\draw (4) to (5);
\draw (3) to (6);
\draw (3) to (7);
\draw (1) circle(0.23);
\draw (2) circle(0.23);
\draw (3) circle(0.23);
\draw (4) circle(0.23);
\draw (5) circle(0.23);
\draw (6) circle(0.23);
\draw (6) circle(0.23);
\draw (7) circle(0.23);
\end{tikzpicture}
\end{center}

\begin{center}
\begin{tikzpicture}[auto]
\node (1) at (1,0) {$0$};
\node (2) at (2,0) {$0$};
\node (3) at (3,0) {$\bullet$};
\node (4) at (4,0) {$0$};
\node (5) at (5,0) {$0$};
\node (6) at (3,1) {$\bullet$};
\node (7) at (3,-1) {$\bullet$};
\node (1-) at (1.3,0.3) {$1$};
\node (2-) at (2.3,0.3) {$2$};
\node (3-) at (3.3,0.3) {$3$};
\node (4-) at (4.3,0.3) {$4$};
\node (5-) at (5.3,0.3) {$5$};
\node (6-) at (3.3,1.3) {$6$};
\node (7-) at (3.3,-1.3) {$7$};
\draw (1) to (2);
\draw (2) to (3);
\draw (3) to (4);
\draw (4) to (5);
\draw (3) to (6);
\draw (3) to (7);
\draw (1) circle(0.23);
\draw (2) circle(0.23);
\draw (3) circle(0.23);
\draw (4) circle(0.23);
\draw (5) circle(0.23);
\draw (6) circle(0.23);
\draw (6) circle(0.23);
\draw (7) circle(0.23);
\end{tikzpicture}
\hspace{1cm}
\begin{tikzpicture}[auto]
\node (1) at (1,0) {$\bullet$};
\node (2) at (2,0) {$\bullet$};
\node (3) at (3,0) {$0$};
\node (4) at (4,0) {$\bullet$};
\node (5) at (5,0) {$\bullet$};
\node (6) at (3,1) {$0$};
\node (7) at (3,-1) {$0$};
\node (1-) at (1.3,0.3) {$1$};
\node (2-) at (2.3,0.3) {$2$};
\node (3-) at (3.3,0.3) {$3$};
\node (4-) at (4.3,0.3) {$4$};
\node (5-) at (5.3,0.3) {$5$};
\node (6-) at (3.3,1.3) {$6$};
\node (7-) at (3.3,-1.3) {$7$};
\draw (1) to (2);
\draw (2) to (3);
\draw (3) to (4);
\draw (4) to (5);
\draw (3) to (6);
\draw (3) to (7);
\draw (1) circle(0.23);
\draw (2) circle(0.23);
\draw (3) circle(0.23);
\draw (4) circle(0.23);
\draw (5) circle(0.23);
\draw (6) circle(0.23);
\draw (6) circle(0.23);
\draw (7) circle(0.23);
\end{tikzpicture}
\end{center}
\caption{Four $\frac13$-eigencomponents used in Example \ref{exam:7-G}.}
\label{fig:7-G-2}
\end{figure}
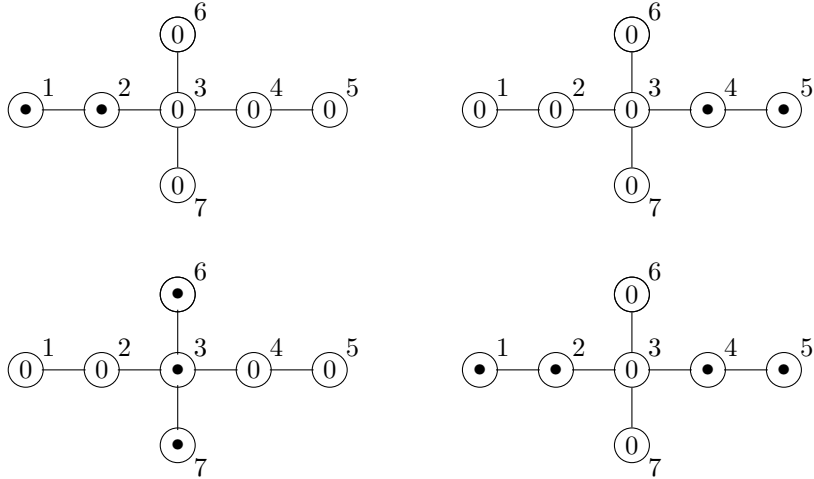

In the first two cases, $\delta_0=9$, while in the next two cases, $\delta_0=6$. Hence, the optimal Cheeger cut are $\{1,2, 4, 5\}$ and $\{3,6,7\}$.

\end{example}

\section{Conclusion}


In this paper, we systematically studied the theory as well as applications of nodal domains of eigenvectors for the graph $1$-Laplacian. All-around detailed comparison between the standard Laplacian and $1$-Laplacian was presented. The main results are summarized below.
\begin{enumerate}
\item We found that the eigenvectors for graph $1$-Laplacian admit the localization property (see Theorem \ref{th:equivalent-binary}), that is, the restrict of an eigenvector to one of its nodal domain is again an eigenvector with the same eigenvalue.
This is a significant difference from the standard graph Laplacian. It also allows us to glue some simple modules into a complex eigenvector by some special techniques, with which the eigenvectors of the most possible nodal domains for path graphs and cycle graphs are obtained (see Theorems \ref{th:nodal-Pn} and \ref{th:nodal-Cn}). These results can be seen as counterparts of Sturm-Liouville oscillation theorem in ODE,  of the oscillatory eigenfunctions for $1$-Laplacian on intervals and circles \cite{Chang2009}, as well as of the oscillatory eigenfunctions for standard Laplacian on path graphs \cite{GantmacherKrein2002}.

\item We extended the Courant nodal domain theorem to the $1$-Laplacian eigenvalues defined by minimax principles (see Theorem \ref{th:strong-nodal-domain}). Actually, the Courant nodal domain theorem
still holds for strong nodal domains of the $1$-Laplacian eigenvectors, but not for weak ones (see Example \ref{exam:10G-weak-nodal}). It should be noted that the lower bound for the number of strong (weak) nodal domains is $1$ by Theorem \ref{th:equivalent-binary}. Further, an inequality connecting the $k$-way Cheeger constant with the critical values is obtained (see Theorem \ref{th:ckhk}).

\item  The number of eigenvalues can be much larger than the number of eigenvalues defined by the minimax principle (see Example \ref{exam:counterexampleG6}), which is another significant difference from the standard Laplacian.


\item The algebraic multiplicity (introduced in Section \ref{sec:algebra}) of a  eigenvalue,
showing big difference from the topological multiplicity, equals to the dimension of the linear space spanned by the set of eigenvectors corresponding to the prescribed eigenvalue,
and no less than the number of nodal domains of the corresponding eigenvector. Moreover, the calculation of the algebraic multiplicity is much easier than that of the topological multiplicity.

\end{enumerate}

\section*{Acknowledgement}
This research was supported by grants from the National Natural Science Foundation of China (Nos.~11371038, 11421101, 11471025, 61121002, 91330110).

\section*{Appendix}
\label{sec:Appendix1}

Without loss of generality, we may assume that $G$ is connected, because otherwise we can replace $G$ by its connected components. Denote by $\#_n$ the largest possible number of different eigenvalues of a graph of order $n$. It can be easily checked that $\#_1=1$, $\#_2=2$ and $\#_3=2$.
In order to prove the minimum order $n$ of graphs which satisfies $\#_n>n$ is $6$, we only need to verify the cases of $|V|\in\{4,5\}$.

\subsection*{$|V|=4$}\label{subsection:|V|=4}

For a connected graph with $4$ vertices, Fig.~\ref{fig:appendix-order4} shows all the cases and related $1$-Laplacian eigenvalues in the sense of graph isomorphism (isomorphic graphs have the same eigenvalue set). That is, $\#_4 = 3$.
\begin{figure}[htpb]
\centering
\begin{tabular}{cccccc}
\hline
\includegraphics[height=2.3cm]{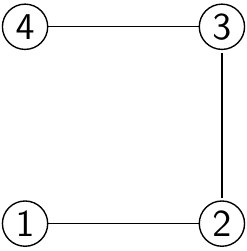}
&
\includegraphics[height=2.3cm]{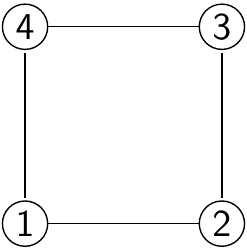}
&
\includegraphics[height=2.3cm]{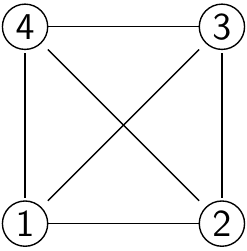}
&
\includegraphics[height=2.3cm]{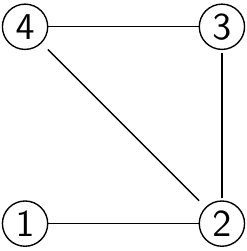}
&
\includegraphics[height=2.3cm]{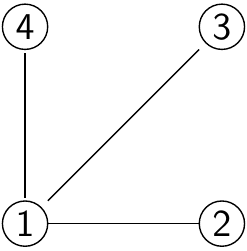}
&
\includegraphics[height=2.3cm]{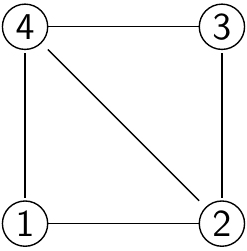}
\\ \hline
 $0<\frac13<1$& $0<\frac12<1$& $0<\frac23<1$& $0<\frac23<1$&$0<1$& $0<\frac35<1$
\\ \hline
\end{tabular}
\caption{Connected graphs with $4$ vertices and their eigenvalues used in Appendix for the case $|V|=4$.}
\label{fig:appendix-order4}
\end{figure}

\subsection*{$|V|=5$}

Now we concentrate on the connected graphs on the vertex set $V=\{1,2,3,4,5\}$. According to Corollary \ref{cor:delta+<delta}, we need to only check
all possible subsets $A\subset V$ such that $(A,E|_A)$ is a connected subgraph of $G$ and $\vol(A)\le \frac12\vol(V)$. Once $\hat{\vec 1}_A$ is an eigenvector, the corresponding eigenvalue is $\mu=I(\hat{\vec 1}_A)$. We just consider the cases of $|A|\in\{2,3,4\}$ because the remaining cases of $|A|=1,5$ are trivial (the eigenvalue is $0$ and $1$ for $|A|=1$ and $5$, respectively).
For the case of $|A|\ge 3$, i.e. $|A^c|\le 2$, it can be verified that  $|\partial A|=|\partial A^c|$ and
\[
\vol(A)\ge 2\times (3-1)+|\partial A|>2\times (2-1)+|\partial A^c|\ge \vol(A^c)
\]
implying $\vol(A)>\frac12\vol(V)$,
and thus we ignore this case.
Hence, it remains only the case of $|A|= 2$.

Without loss of generality, we may suppose $A=\{1,2\}$. By Lemma \ref{Lem:minor}, $(A^c,E|_{A^c})$ is connected (otherwise there exists $i_0\in A^c$ such that $j\in A$ whenever $j\sim i_0$ and thus $\hat{\vec 1}_A$ can not be an eigenvector), then $\vol(A)=d_1+d_2\in \{3,4,5,6,7,8\}$ and $1\leq d_i\leq 4$ with $i=1,2$.

\begin{enumerate}[(Q1)]
\item $d_1+d_2=3$, \ie $(d_1,d_2)=(1,2),(2,1)$. All possible cases and related $1$-Laplacian eigenvalues in the sense of graph isomorphism are shown in Fig.~\ref{fig:appendix-order5-1}.
We can see there that $\#_5 = 4$ in this case.
\begin{figure}[htpb]
\centering
\begin{tabular}{ccc}
\hline
\includegraphics[height=2.3cm]{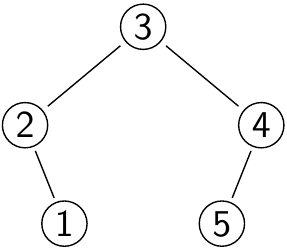}
&
\includegraphics[height=2.3cm]{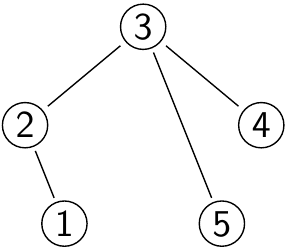}
&
\includegraphics[height=2.3cm]{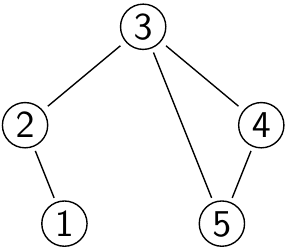}
\\ \hline
$0<\frac13<1$&$0<\frac13<1$ & $0<\frac13<\frac12<1$
\\ \hline
\end{tabular}
\caption{Graphs and their spectrums used in Case (Q1).}
\label{fig:appendix-order5-1}
\end{figure}

\item $d_1+d_2\in\{7,8\}$, \ie $(d_1,d_2)=(3,4),(4,3),(4,4)$.
All possible cases and related $1$-Laplacian eigenvalues in the sense of graph isomorphism are shown in Fig.~\ref{fig:appendix-order5-2}.
We can see there that $\#_5 = 4$ in this case.
\begin{figure}[htpb]
\centering
\begin{tabular}{cccc}
\hline
\includegraphics[height=2.3cm]{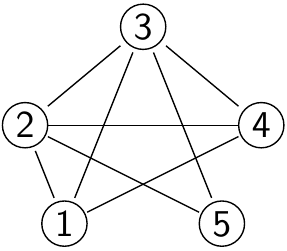}
&
\includegraphics[height=2.3cm]{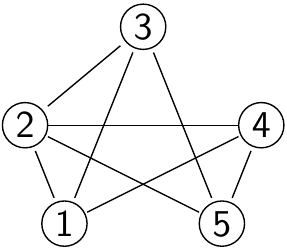}
&
\includegraphics[height=2.3cm]{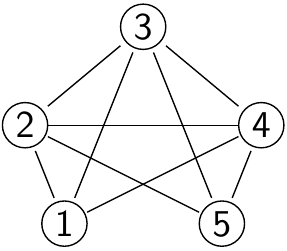}
&
\includegraphics[height=2.3cm]{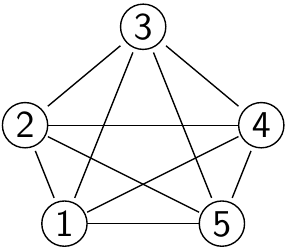}
\\ \hline
 $0<\frac23<\frac57<1$& $0<\frac23<\frac57<1$ & $0<\frac57<\frac34<1$& $0<\frac34<1$
\\ \hline
\end{tabular}
\caption{Graphs and their spectrums used in Case (Q2).}
\label{fig:appendix-order5-2}
\end{figure}

\item $d_1+d_2\in\{4,5,6\}$. We have
$$
I(\hat{\vec 1}_A) = \frac{|\partial A|}{\vol(A)}=\frac{d_1-1+d_2-1}{d_1+d_2}=1-\frac{2}{\vol(A)}\in\left\{1-\frac24,1-\frac25,1-\frac26\right\}
=\left\{\frac12,\frac35,\frac23\right\}.
$$
In consequence, all the $1$-Laplacian eigenvalues in this case must be contained in $\{0,\frac12,\frac35,\frac23,1\}$. That is, $\#_5\leq 5$.
\end{enumerate}

Hence, we have
\begin{cor}
Let $G=(V,E)$ be a graph of order $n\leq 5$. Then the number of $1$-Laplacian eigenvalues is at most $n$.
\end{cor}

In fact, we are able to further prove that $\#_5=5$ can not be true even for the cases of $d_1+d_2\in\{4,5,6\}$.
Otherwise if $\#_5=5$ holds, according to the discussions of Case (X3), there exists a graph $G=(V,E)$ with $V=\{1,2,3,4,5\}$ such that
\begin{equation}\label{eq:D}
D:=\{d_i+d_j: \{i,j\}\in E\}=\{4,5,6\}.
\end{equation}
Then, we may assume that $\vol(\{1,2\})=4$, which could be reduced to
$(d_1,d_2)=(1,3)$, $(2,2)$, $(3,1)$.

\begin{enumerate}
\item[(Q4)] $(d_1,d_2)=(1,3),(3,1)$.
All possible graphs in the sense of graph isomorphism and related sets $D$ (the second row) as well as the $1$-Laplacian eigenvalues (the third row) are shown in Fig.~\ref{fig:appendix-order5-3}. We can see there that only first two
graphs satisfy the constraint \eqref{eq:D}, but
the number of different eigenvalues is $4$ and $3$, respectively. This is a contradiction.
\begin{figure}[htpb]
\centering
\begin{tabular}{cccc}
\hline
\includegraphics[height=2.3cm]{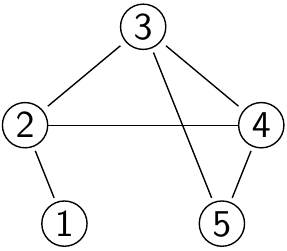}
&
\includegraphics[height=2.3cm]{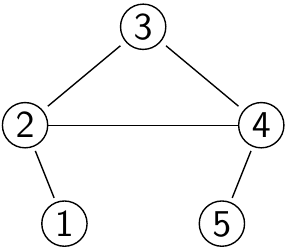}
&
\includegraphics[height=2.3cm]{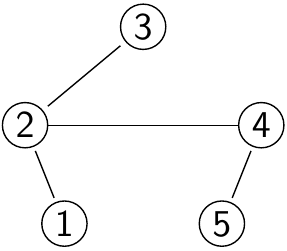}
&
\includegraphics[height=2.3cm]{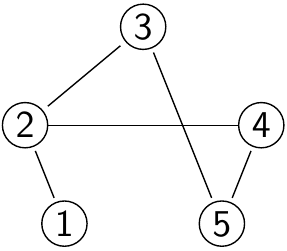}
\\ \hline
$\{4,5,6\}$ &
$\{4,5,6\}$ & $\{3,4,5\}$ & $\{4,5\}$  \\
\hline
 $0<\frac12<\frac35<1$ &
 $0<\frac12<1$ & $0<\frac13<1$&  $0<\frac12<1$
\\ \hline
\end{tabular}
\caption{Graphs and their spectrums used in Case (Q4).}
\label{fig:appendix-order5-3}
\end{figure}

\item[(Q5)] $(d_1,d_2)=(2,2)$. All possible graphs in the sense of graph isomorphism and related sets $D$ (the second and fourth rows) as well as the $1$-Laplacian eigenvalues (the third and sixth rows) are shown Fig.~\ref{fig:appendix-order5-4}. We can see there that only first two
graphs satisfy the constraint \eqref{eq:D}, but
the number of different eigenvalues is $4$ and $3$, respectively. This is a contradiction.
\begin{figure}[htpb]
\centering
\begin{tabular}{cccc}
\hline
\includegraphics[height=2.3cm]{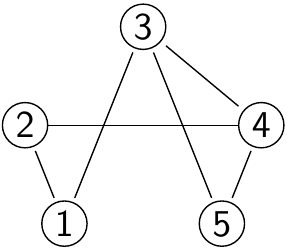}
&
\includegraphics[height=2.3cm]{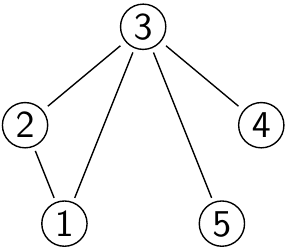}
&
\includegraphics[height=2.3cm]{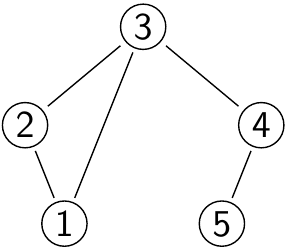}
&
\includegraphics[height=2.3cm]{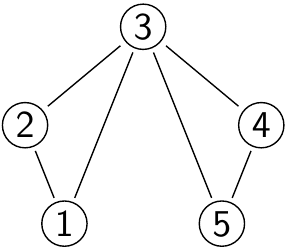}
\\
\hline
$\{4,5,6\}$
&
$\{4,5,6\}$ & $\{3,4,5\}$ & $\{4,6\}$
\\
\hline
$0<\frac12<\frac35<1$
& $0<\frac12<1$&$0<\frac13<\frac12<1$ & $0<\frac12<1$
\\ \hline
\includegraphics[height=2.3cm]{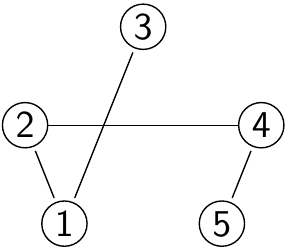}
&
\includegraphics[height=2.3cm]{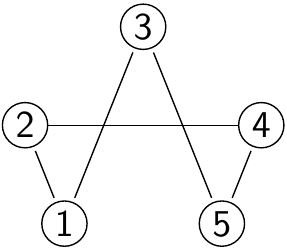}
&
\includegraphics[height=2.3cm]{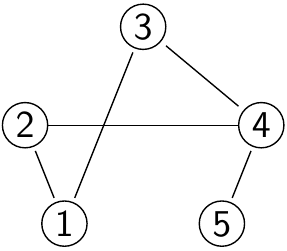}
&
\\
\hline
$\{3,4\}$ & $\{4\}$ & $\{4,5\}$ &
\\
\hline
$0<\frac13<1$& $0<\frac12<1$ & $0<\frac12<1$ &
\\
\hline
\end{tabular}
\caption{Graphs and their spectrums used in Case (Q5).}
\label{fig:appendix-order5-4}
\end{figure}
\end{enumerate}

Combining all the results of Cases (Q1)-(Q5), we have $\#_5=4$.



\end{document}